\documentclass[bj,numbers]{imsart}

\usepackage{amsmath, amssymb, amsthm, mathtools, enumitem}
\usepackage{accents, xcolor}
\usepackage{cancel}
\usepackage{amssymb}
\usepackage{amsthm}
\usepackage{amsfonts}
\usepackage{amsmath}
\usepackage{pmboxdraw}
\usepackage{verbatim} 
\usepackage{graphicx}
\usepackage{color}
\usepackage{colortbl}
\usepackage[normalem]{ulem}
\usepackage{subcaption}
\usepackage{todonotes}
\usepackage{kantlipsum}
\allowdisplaybreaks
\usepackage{bbm}
\usepackage{tikz}
\usepackage{pgfplots} 
\pgfplotsset{compat=1.18} 
\usetikzlibrary{decorations, arrows.meta, patterns}

\startlocaldefs
\theoremstyle{plain}
\newtheorem{theorem}{Theorem}[section]
\newtheorem{lemma}[theorem]{Lemma}

\newtheorem{proposition}[theorem]{Proposition}

\theoremstyle{definition}
\newtheorem{rem}{Remark}

\newcommand{\Li}{\operatorname{Li}}
\newcommand{\re}{\operatorname{Re}}

\newcommand{\im}{\operatorname{Im}}

\newcommand{\Tr}{\operatorname{Tr}}

\def\C{\mathbb{C}}

\def\E{\mathbb{E}}

\def\P{\mathbb{P}}
\def\R{\mathbb{R}}

\endlocaldefs

\begin{document}

\begin{frontmatter}
\title{Moderate-to-large deviation asymptotics for \\real eigenvalues of the elliptic Ginibre matrices}
\runtitle{From moderate to large deviation of real eigenvalue statistics}

\begin{aug}
\author[A]{\inits{S.-S.}\fnms{Sung-Soo}~\snm{Byun}\ead[label=e1]{sungsoobyun@snu.ac.kr}}
\author[B]{\inits{J.}\fnms{Jonas}~\snm{Jalowy}\ead[label=e2]{jjalowy@math.upb.de}}
\author[C]{\inits{Y.-W.}\fnms{Yong-Woo}~\snm{Lee}\ead[label=e3]{hellowoo@snu.ac.kr}}
\author[D]{\inits{G.}\fnms{Gr{\'e}gory}~\snm{Schehr}\ead[label=e4]{schehr@lpthe.jussieu.fr}}
\address[A]{Department of Mathematical Sciences and Research Institute of Mathematics, Seoul National University, Seoul 151-747, Republic of Korea\printead[presep={,\ }]{e1}}

\address[B]{Institute of Mathematics, Paderborn University, Warburger Str. 100, 33098 Paderborn, Germany\printead[presep={,\ }]{e2}}

\address[C]{Department of Mathematical Sciences, Seoul National University, Seoul 151-747, Republic of Korea\printead[presep={,\ }]{e3}}

\address[D]{Sorbonne Universit\'e, Laboratoire de Physique Th\'eorique et Hautes Energies, CNRS UMR 7589, 4 Place Jussieu, 75252 Paris Cedex 05, France\printead[presep={,\ }]{e4}}
\end{aug}

\begin{abstract}
We study the statistics of the number of real eigenvalues in the elliptic deformation of the real Ginibre ensemble. As the matrix dimension grows, the law of large numbers and the central limit theorem for the number of real eigenvalues are well understood, but the probabilities of rare events remain largely unexplored. Large deviation type results have been obtained only in extreme cases---when either a vanishingly small proportion of eigenvalues are real or almost all eigenvalues are real. Here, in both the strong and weak asymmetry regimes, we derive the probabilities of rare events in the moderate-to-large deviation regime, thereby providing a natural connection between the previously known regime of Gaussian fluctuations and the large deviation regime. Our results are new even for the classical real Ginibre ensemble.
\end{abstract}

\begin{keyword}
\kwd{Large deviations}
\kwd{Moderate deviations}
\kwd{Real eigenvalues}
\kwd{Real elliptic Ginibre ensemble}
\end{keyword}

\end{frontmatter}


\pgfmathdeclarefunction{gauss}{2}{%
    \pgfmathparse{1/(#2*sqrt(2*pi))*exp(-((x-#1)^2)/(2*#2^2))}%
}

\section{Introduction}
 
Probabilistic limit theorems, which are central to modern probability theory and its applications---such as in statistical physics---provide a unifying framework for understanding fluctuations in complex systems across different scales. For a system of size $n$ (for example, the sum of $n$ independent random variables), the central limit theorem (CLT) typically yields Gaussian laws with fluctuations of order $n^{1/2}$, while the large deviation principle (LDP) characterises rare events of order $n$ with exponential precision. Between these two regimes lies the \emph{moderate deviation scale}, in which fluctuations exceed the CLT scale but remain sublinear in $n$. Results in this regime refine the Gaussian approximation and capture the onset of large deviation behaviour.  

In random matrix theory, while central limit theorems for global statistics are well established and large deviation methods capture extreme fluctuations, systematic results in the intermediate regime—the moderate deviation principles—have received comparatively little attention. 
This intermediate scaling behaviour has recently been investigated in the context of the complex Ginibre ensemble, where it occurs in the statistics of the eigenvalue with the largest modulus \cite{LGMS18} and in eigenvalue counting statistics \cite{LGCCKMS19,LMS19}, i.e. the number of eigenvalues within a prescribed region of the complex plane, see also \cite{Ch23,Ch22,ABE23,ACCL24,FKP24,CESX22,XZ24,BP26,MMO25,GPX24,LS25,MM25,HM25} and references therein for recent work on related variants. In this work, we address this regime for a fundamental observable in non-Hermitian random matrix theory: \emph{the number of real eigenvalues in asymmetric random matrices}.

The statistics of real eigenvalues of non-Hermitian random matrices have been studied extensively in the literature. In particular, the number of real eigenvalues has been analysed in a variety of models, most notably the real Ginibre ensemble \cite{EKS94,FN07,Si17a,KPTTZ15} and its extensions. Prominent examples include the elliptic Ginibre ensembles \cite{FN08,BKLL23,BL24,Fo24}, spherical Ginibre matrices \cite{FM12,EKS94,Fo25}, the truncated orthogonal ensemble \cite{FIK20,FK18,LMS22,GP19}, products of GinOE matrices \cite{AB23,Si17a,Fo14,FI16, FS23a}, and asymmetric Wishart matrices \cite{BN25}. 

Real eigenvalues of non-Hermitian random matrices, such as those in the real Ginibre ensemble, play a central role in physics, particularly in characterizing the spectral and dynamical properties of complex systems. Their number and distribution provide valuable insights into phase transitions, symmetry classes, and the onset of localization, thereby reflecting fundamental features of open and dissipative quantum systems (see, e.g. \cite{BB98,XKLOS22}). Beyond their role as spectral diagnostics, real eigenvalues are also crucial for assessing the stability of complex systems.  
For example, they play a key role in determining the number of equilibrium points in non-relaxational dynamical systems~\cite{BFK21,Fyo16,Kiv24}, which generalize May’s classical model of ecological networks~\cite{May72}. 
Similarly, in neural networks, they correspond to non-oscillatory relaxation modes that govern the network’s response to perturbations and overall dynamical stability \cite{RA06,SCSS88}. Finally, very interesting connections have been unveiled between the statistics of the real eigenvalues in the real Ginibre ensemble and annihilating Brownian particles on the line (see e.g. \cite{GPTZ18}). 
 
We now introduce our object of study.  
By definition, the real Ginibre ensemble (GinOE) is the ensemble of $n \times n$ matrices $G$ with i.i.d.\ standard normal entries, see \cite{BF25} for a review.  
Its elliptic extension, designed to interpolate between GinOE and the symmetric Gaussian orthogonal ensemble (GOE), is known as the elliptic GinOE (eGinOE). It is defined by  
\begin{equation*} 
    X := \frac{\sqrt{1+\tau}}{2}(G+G^T)+\frac{\sqrt{1-\tau}}{2}(G-G^T),
\end{equation*}
where \( \tau \in [0,1) \) is the asymmetry parameter. In particular, $\tau=0$ recovers GinOE, while the limit $\tau \to 1$ yields the classical Gaussian orthogonal ensemble (GOE) \cite{Fo10}. It is well known, under the name of the \emph{elliptic law}, that as $n \to \infty$, the empirical spectral measure of the eigenvalues of $X$ converges to the uniform measure supported on the ellipse 
\begin{equation*} 
\Big\{ z \in \C: \Big( \frac{\re z}{1+\tau}\Big)^2 + \Big(\frac{\im z}{1-\tau}\Big)^2 \le 1 \Big \};
\end{equation*}
see Figure \ref{Fig_realisations}.

We begin with a brief overview of the known results on the statistics of the number of real eigenvalues in the eGinOE.
In the study of the eGinOE, there are two distinct asymptotic regimes:
\begin{itemize}
    \item \textbf{Strong Asymmetry}: the parameter $\tau \in [0,1)$ is kept 
    fixed as $n \to \infty$, see Figure \ref{Fig_realisations} (a).  
    \smallskip 
    \item \textbf{Weak Asymmetry} \cite{FKS97,FKS98}: the parameter $\tau \equiv \tau_n \to 1$ at a prescribed rate, see Figure \ref{Fig_realisations} (b). We focus on the scaling
    $\tau = 1 - \alpha^2/n$ for fixed $\alpha \in (0,\infty)$. 
Such a scaling arises naturally from the geometry of the elliptic law. 
The support of the limiting density has area $\mathcal{A} = \pi(1-\tau^2),$ and hence the typical interparticle spacing scales as $\Delta z_{\mathrm{inter}} \sim \sqrt{\mathcal{A}/n}.$ 
A crossover occurs when the vertical scale of the support, of order $1-\tau$, becomes comparable to this spacing, which yields $1-\tau = O(1/n).$ 
In this regime, the eigenvalues concentrate near the real axis while exhibiting fluctuations in the imaginary direction on the same scale as the typical spacing. 
Consequently, the statistics of real eigenvalues display a critical transition between those of symmetric and asymmetric random matrices.
\end{itemize}

\begin{figure}[t]
    \begin{subfigure}{6.8cm}
        \begin{center}
            \includegraphics[width=6.8cm]{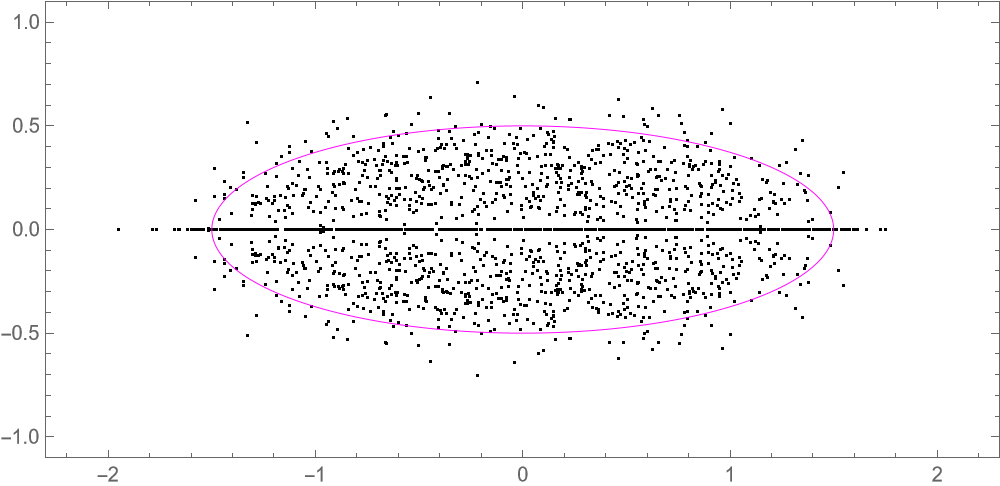}
        \end{center}
        \subcaption{ $\tau=1/2$\, (strong asymmetry). }  \label{Fig_realisation_SA}
    \end{subfigure}
    \qquad
    \begin{subfigure}{6.8cm}
        \begin{center}
            \includegraphics[width=6.8cm]{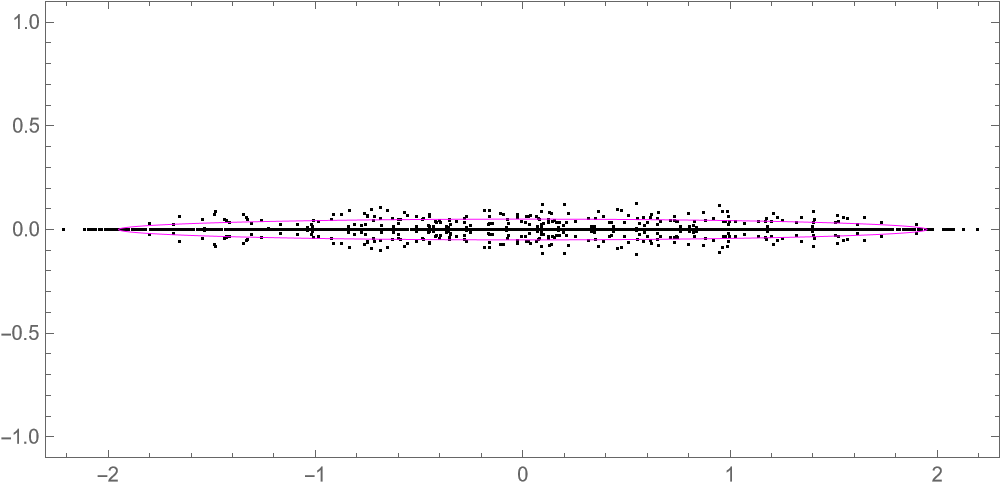}
        \end{center}
        \subcaption{ $\tau=19/20$ \, (weak asymmetry).  } \label{Fig_realisation_WA}
    \end{subfigure}
    \caption{Eigenvalues of $100$ realisations of the eGinOE with $n=20$ for a given $\tau$.}
    \label{Fig_realisations}
\end{figure}

Let $\mathcal{N}_n$ denote the number of real eigenvalues of the eGinOE 
matrix $X$. We denote by $\P$ and $\E$ the probability measure and expectation associated with $\mathcal{N}_n$, respectively.  
We first discuss the typical behaviour of the random variable $\mathcal{N}_n$. Here we emphasise that the results in most of the literature below (e.g.~\cite{BMS25,BKLL23,FN08,KPTTZ15}) were obtained for even dimensions~$n$, see however \cite{KA2005,FM09,Si07,Si09} for the analysis for odd dimensions.  
Nevertheless, it is believed that the leading asymptotic behaviour remains the same for odd dimensions, as in the case $\tau=0$ for $\E \mathcal{N}_{n}$ and for the scaling limits of the correlation functions. We refer the interested reader to the end of Section~\ref{sec. Exponential profiles} for a further discussion in this direction.

\begin{itemize}
    \item For the strong asymmetry regime, where $\tau$ is fixed, it has been shown that both the mean and the variance of $\mathcal{N}_n$ grow at order $O(\sqrt{n})$. More precisely, 
    \begin{equation} \label{eq. mean and variance sH}
     \lim_{ n \to \infty } \frac{ \mathbb{E} \mathcal{N}_n }{ \sqrt{n} } = \sqrt{\frac{2}{\pi} \frac{1+\tau}{1-\tau}  }, \qquad   \lim_{n \to \infty} \frac{ \textup{Var}\, \mathcal{N}_n }{ \sqrt{n} } = (2-\sqrt{2})   \sqrt{\frac{2}{\pi} \frac{1+\tau}{1-\tau}  } . 
    \end{equation}
    These were established in \cite{BKLL23,FN08}, and for the special case $\tau=0$ (the GinOE), these results had appeared in earlier works \cite{EKS94,FN07}.
    \smallskip 
    \item For the weak asymmetry regime, where $\tau = 1 - \alpha^2/n$, the growth is of a larger order $O(n)$. In this case, it was shown in \cite{BKLL23} that   
     \begin{equation} \label{eq. mean and variance wH}
  \lim_{ n \to \infty } \frac{ \mathbb{E} \mathcal{N}_n }{ n } = c(\alpha), \qquad  \lim_{n \to \infty} \frac{ \textup{Var}\, \mathcal{N}_n }{ n  } =  2 \Big( c(\alpha)-   c(\sqrt{2}\alpha) \Big). 
\end{equation} 
Here, \begin{equation} \label{eq. coeff avg number} 
   c(\alpha) := e^{-\frac{\alpha^2}{2}} \Big( I_0\big(\tfrac{\alpha^2}{2}\big) + I_1\big(\tfrac{\alpha^2}{2}\big) \Big),  
\end{equation}
where $I_\nu$ is the modified Bessel function of the first kind \cite[Chapter 10]{NIST}. 
\end{itemize}
 
Beyond the mean and variance of $\mathcal{N}_n$, a central limit theorem has also been established for both the strong and weak asymmetry regimes. More precisely, it was shown in \cite{BMS25,Fo24,Si17} that  
\begin{equation*}
\frac{\mathcal N_n - \mathbb E \mathcal N_n  }{\sqrt{ \textup{Var} \mathcal N_n } }  \to \textup{N}(0,1), \qquad  \textup{as } n \to \infty,
\end{equation*}
where $\mathrm{N}(0,1)$ denotes the standard normal distribution.  
See Figure~\ref{Fig_illustration of pnm} for an illustration of the typical Gaussian fluctuation regime.

In addition to the number of real eigenvalues, the averaged real eigenvalue density and its asymptotic behaviour were also derived in \cite{EKS94,FN08} for the strong asymmetry regime and in \cite{BKLL23,Efe97} for the weak asymmetry regime. The precise convergence rates, often called the finite-size corrections, were further obtained in \cite{BL24}. These asymptotic behaviours of the real eigenvalue density play an important role in the study of the complexity of random landscapes arising in non-gradient flows \cite{Fyo16,BFK21}.

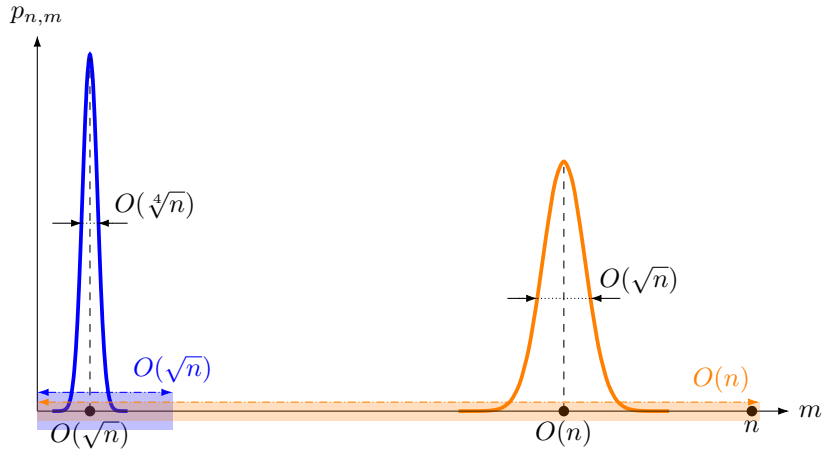
\begin{figure}
    \begin{center}
        \begin{tikzpicture}[scale=1, transform shape]
            \begin{scope}[scale=1]
                \begin{axis}[
                    anchor=origin,
                    x=0.1cm,
                    at={(0.7cm,0)},
                    style={samples=51,smooth,color=blue,line width=0.5mm},
                    hide axis
                ]
                \addplot[mark=none] {gauss(0,1)};
                \end{axis}
    
                \draw[dashed] (0.7, 0) -- (0.7, 4.7);
                
                \fill[black] (0.7, 0) circle (2pt) node[below] {$O(\sqrt{n})$};
    
                \draw[-, -{Latex[length=1.5mm, width=1mm]}] (1.2, 2.5) -- (0.8, 2.5);
                \draw[-, -{Latex[length=1.5mm, width=1mm]}] (0.2, 2.5) -- (0.6, 2.5);
                \node[right] at (0.9, 2.75) {$O(\sqrt[4]{n})$};
                \draw[densely dotted] (0.6, 2.5) -- (1,2.5);

                \begin{axis}[
                    anchor=origin,
                    x=0.4cm,
                    at={(7cm,0)},
                    style={samples=51,smooth,color=orange,line width=0.5mm},
                    hide axis,
                    scale=0.7
                ]
                \addplot[mark=none] {gauss(0,1)};
                \end{axis}
    
                \draw[-, -{Latex[length=1.5mm, width=1mm]}] ({1.8+5.95}, 1.5) -- ({1.4+5.95}, 1.5);
                \draw[densely dotted] ({0.9+5.5}, 1.5) -- ({1.5+5.95},1.5);
                \draw[-, -{Latex[length=1.5mm, width=1mm]}] ({0.6+5.65}, 1.5) -- ({1+5.65}, 1.5);
                \node[right] at ({1.4+5.95}, 1.75) {$O(\sqrt{n})$};
                
                \draw[dashed] ({1.2+5.8}, 0) -- ({1.2+5.8}, 3.25);
                
                \fill[black] ({1.2+5.8}, 0) circle (2pt) node[below] {$O(n)$};

                \fill[black] (9.5, 0) circle (2pt) node[below] {$n$};
                \draw[-{Latex[length=1.5mm, width=1mm]}] (0,0) -- (0, 5) node[above] {$p_{n,m}$};
                \draw[-{Latex[length=1.5mm, width=1mm]}] (0,0) -- (10, 0) node[right] {$m$};

    
                \fill[nearly transparent, blue] (0,-0.25) rectangle (1.8,0.25);
                \draw[blue,{Latex[length=1.5mm, width=1mm]}-{Latex[length=1.5mm, width=1mm]},densely dashdotted] (0,0.25) -- (1.8, 0.25) node[above] {$O(\sqrt{n})$};
    
                \fill[nearly transparent, orange] (0,-0.12) rectangle (9.6,0.12);
                \draw[orange,{Latex[length=1.5mm, width=1mm]}-{Latex[length=1.5mm, width=1mm]},densely dashdotted] (0,0.12) -- (9.6, 0.12) node[above left] {$O(n)$};
            \end{scope}
        \end{tikzpicture}
    \end{center}
    \caption{A schematic sketch of $m\mapsto p_{n,m}$. The graph on the left (blue) illustrates the case with strong asymmetry, while the graph on the right (orange) illustrates the case with weak asymmetry. The regimes studied in this paper are depicted on the $x$-axis as shaded regions.} \label{Fig_illustration of pnm}
\end{figure}

We now turn to the rare events and the asymptotic behaviour of their probabilities. For this, we denote  
\begin{equation}
    p_{n,m} := \P(\mathcal{N}_n = m)
\end{equation} 
for the probability that the eGinOE matrix has exactly $m$ real eigenvalues. 
When discussing tail probabilities, it is important to distinguish between the left and right tails. The left tail corresponds to the probability of having an exceptionally small number of real eigenvalues compared to the typical value, while the right tail corresponds to the probability of having an exceptionally large number.

\begin{itemize}
    \item We first consider the strong asymmetry regime with $\tau \in [0,1)$. 
    \begin{itemize}
        \item[--]  In the left tail, corresponding to the case  $m = O(\sqrt{n}/\log n)$, it was shown in \cite{BMS25} that  
        \begin{equation} \label{eq_left tail strong}
      \lim_{n \to \infty}  \frac{ \log p_{n,m} }{ \sqrt{n} } =  - \sqrt{\frac{1+\tau}{1-\tau}} \frac{1}{\sqrt{2\pi}} \zeta(\tfrac{3}{2}) \;,
        \end{equation}
        where $\zeta$ is the Riemann zeta function. The GinOE case 
        $(\tau=0)$ was previously obtained in \cite{KPTTZ15}.  
        \smallskip 
        \item[--]  In the right tail, we consider $m=n-O(1)$. 
        In this case, it was shown in \cite{ABL25} that  
        \begin{equation}  \label{eq_right tail strong}
         \lim_{n \to \infty}  \frac{ \log p_{n,m} }{ n^2 } =  -\frac{1}{4} \log\Big( \frac{2}{1+\tau} \Big). 
        \end{equation}
        The special cases $m=n-2$ and $m=n$ can also be found in \cite{Ed97,FN08,AK07}. Moreover, \cite{ABL25} provides the 
        precise correction terms beyond the leading order.  
    \end{itemize}
Note that the $O(n^2)$ decay in the right tail is significantly faster than the $O(\sqrt{n})$ decay in the left tail. This is intuitively natural, since the typical mean of $m$ is of order $O(\sqrt{n})$, which indicates that the distribution is more heavily skewed to the left, see Figure~\ref{Fig_illustration of pnm}.
\smallskip 
\item Next, we consider the weak asymmetry regime with $\tau = 1-\alpha^2/n$.  
\begin{itemize}
    \item[--] In the left tail, corresponding to the case $m=O(n/\log n)$, 
    it was shown in \cite{BMS25} that  
     \begin{equation}  \label{eq_left tail weak}
      \lim_{n \to \infty}  \frac{ \log p_{n,m} }{ n }  \leq \frac{2}{\pi} \int_0^1 \log( 1- e^{-\alpha^2 s^2}) \sqrt{1-s^2} \, ds\;,
    \end{equation}
   independently of $m$. It is conjectured that this inequality is in fact an equality, but due  to a technical issue in the proof of \cite{BMS25}, only the inequality  was established.  
    \smallskip 
    \item[--]  In the right tail, when $m=n-O(1)$, it was shown 
    in \cite{ABL25} that  
    \begin{equation}  \label{eq_right tail weak}
       \lim_{n \to \infty}  \frac{ \log p_{n,m} }{ n }  = -\frac{\alpha^2}{8}.  
    \end{equation}
     The precise corrections and asymptotic expansions were also obtained in \cite{ABL25}.  
 \end{itemize}
Note that, in contrast to the strong asymmetry regime, both the left and 
right tails here exhibit decay of order $O(n)$.  
\end{itemize}

From the literature discussed so far, the typical value and its Gaussian fluctuations are well understood, and the probabilities of observing atypically small numbers of real eigenvalues, such as $m = O(\mathbb{E}\mathcal{N}_n / \log n)$, or extremely large ones, $m = n - O(1)$, have also been investigated. In addition, in the strong asymmetry regime---most notably for the real Ginibre ensemble ($\tau = 0$)---even the macroscopic regime $m = O(n)$ has been partially analysed using Coulomb gas techniques \cite{MPTW16}, see also \cite{ATW14} and Remark~\ref{Rem_strong left right}. 
 
In contrast, two regimes remain largely unexplored: the intermediate regime $m = O(\sqrt{n})$ in the strongly asymmetric case, and the macroscopic regime $m = O(n)$ in the weakly asymmetric setting. In this work, we address both of these regimes, thereby extending the CLT scale and establishing a connection with the extreme large deviation regimes.

\section{Main results and discussions}

\subsection{Main results}

In this section, we present our main results concerning the asymptotic behaviour of $p_{n,m}$. We begin by recalling that, for $s \in \C$, the polylogarithm $\Li_s(z)$ is defined by  
\begin{equation} \label{def of polylog}
\Li_s(z) = \sum_{ k=1 }^\infty \frac{z^k}{k^s},
\end{equation}
where the series converges absolutely for $|z|<1$. The function $\Li_s(z)$ admits an analytic continuation to the complex $z$-plane, with a branch cut typically taken along $[1,\infty)$, see \cite[Section~25.12]{NIST} for details.

We are now ready to state our main results.

\begin{theorem}[\textbf{Moderate and large deviation probabilities}] \label{Thm_main} $ $
\begin{itemize}
    \item[\textup{(i)}] \textup{\textbf{(Strong asymmetry regime)}} Let $\tau \in [0,1)$ be fixed. 
    Suppose that $2m / \mathbb{E}\mathcal{N}_{2n} \to x \in (0,\infty)$ as $n \to \infty$.  
Then we have  
\begin{equation} \label{eqn for mod dev universal form}
\lim_{n\to\infty} \frac{ \log p_{2n,2m} }{ \mathbb{E} \mathcal{N}_{2n} }=   -\frac12  \sup_{u \in \R} \Big\{   xu +    \Li_{3/2}(1-e^u) \Big\}. 
\end{equation}
   Here, $\Li$ is the polylogarithm \eqref{def of polylog} and recall that $\mathbb{E} \mathcal{N}_{2n} \sim  \sqrt{2n} \sqrt{\frac{2}{\pi} \frac{1+\tau}{1-\tau}  } $ by \eqref{eq. mean and variance sH}.  
   \smallskip 
   \item[\textup{(ii)}] \textup{\textbf{(Weak asymmetry regime)}}   Let $   \tau = 1 - \frac{\alpha^2}{2n}$ with fixed $\alpha \in (0,\infty)$.
        Suppose that $m/n \to x \in (0,1)$ as $n\to\infty$.
        Then, we have
        \begin{equation} \label{eq. def phi wA rate function}
            \lim_{n\to\infty} \frac{ \log p_{2n,2m} }{2n}   =   -\frac12 \sup_{u \in \R} \bigg\{ xu - \frac{4}{\pi} \int_0^1 \log\Big( 1- (1-e^u)e^{-\alpha^2 s^2}\Big) \sqrt{1-s^2} \, ds \bigg\}. 
        \end{equation}
\end{itemize}
\end{theorem}

See Figure~\ref{Fig_numerics of pnm} for the numerical verification of Theorem~\ref{Thm_main}.

As previously mentioned, the same statement is expected to hold in the odd-dimensional case. However, several technical obstacles remain to be overcome; see the end of Section~\ref{sec. Exponential profiles} for further discussion.

\begin{figure}[b]
    \begin{subfigure}{4.6cm}
        \begin{center}
            \includegraphics[width=4.6cm]{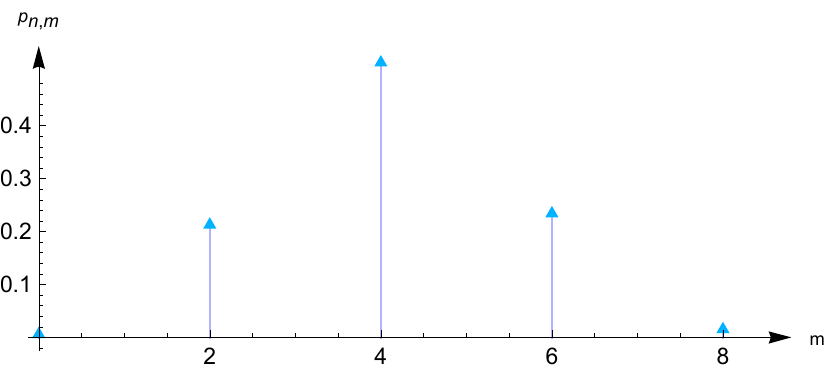}
        \end{center}
        \subcaption{ Strong; $n=8$ }
    \end{subfigure} 
   \begin{subfigure}{4.6cm}
        \begin{center}
            \includegraphics[width=4.6cm]{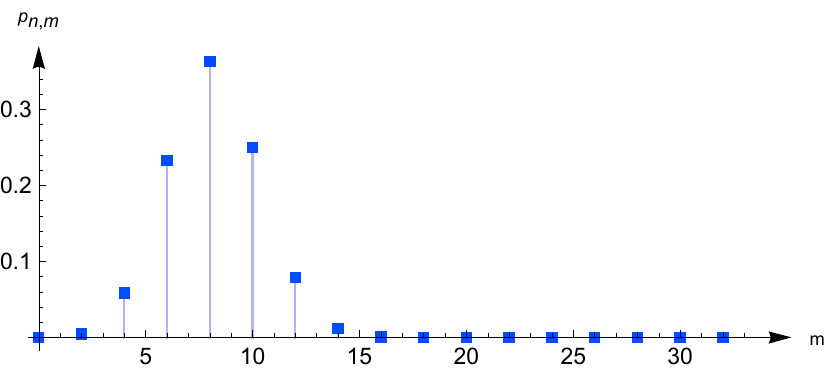}
        \end{center}
        \subcaption{ Strong; $n=32$ }
    \end{subfigure} 
    \begin{subfigure}{4.6cm}
        \begin{center}
            \includegraphics[width=4.6cm]{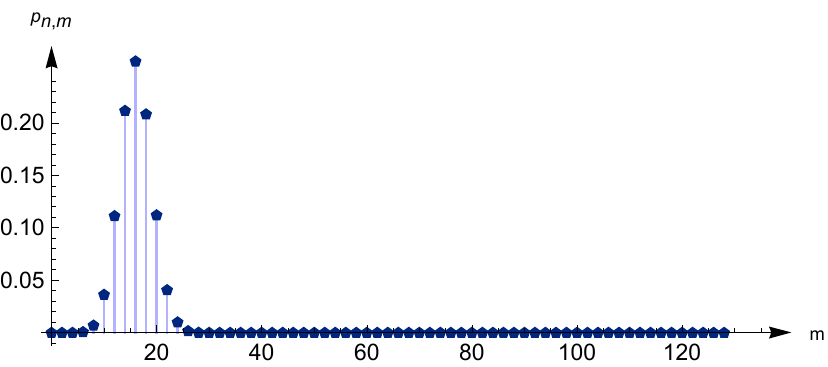}
        \end{center}
        \subcaption{ Strong; $n=128$ }
    \end{subfigure}

      \begin{subfigure}{4.6cm}
        \begin{center}
            \includegraphics[width=4.6cm]{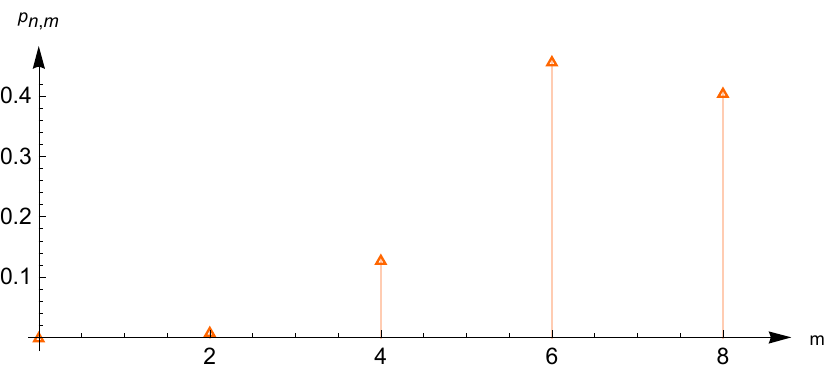}
        \end{center}
        \subcaption{ Weak; $n=8$ }
    \end{subfigure} 
   \begin{subfigure}{4.6cm}
        \begin{center}
            \includegraphics[width=4.6cm]{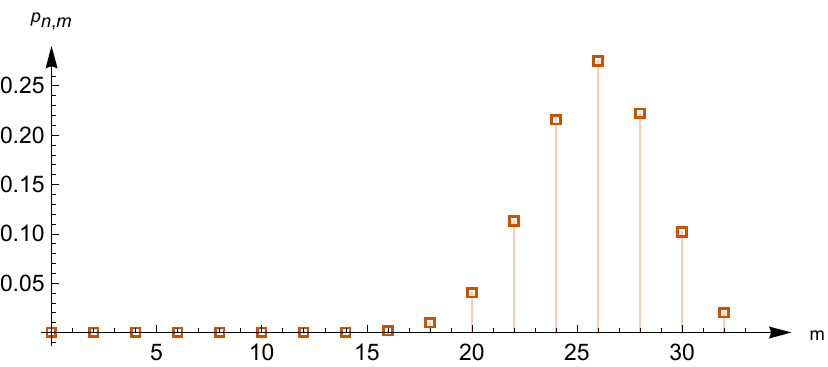}
        \end{center}
        \subcaption{ Weak; $n=32$ }
    \end{subfigure} 
    \begin{subfigure}{4.6cm}
        \begin{center}
            \includegraphics[width=4.6cm]{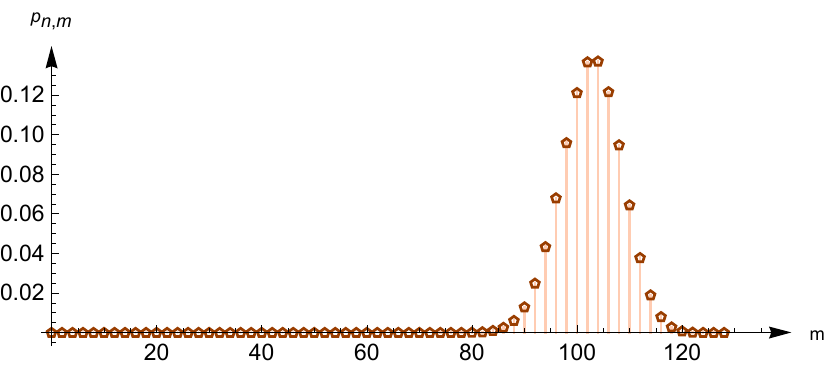}
        \end{center}
        \subcaption{ Weak; $n=128$ }
    \end{subfigure}
    
    \begin{subfigure}{6.8cm}
        \begin{center}
            \includegraphics[width=6.8cm]{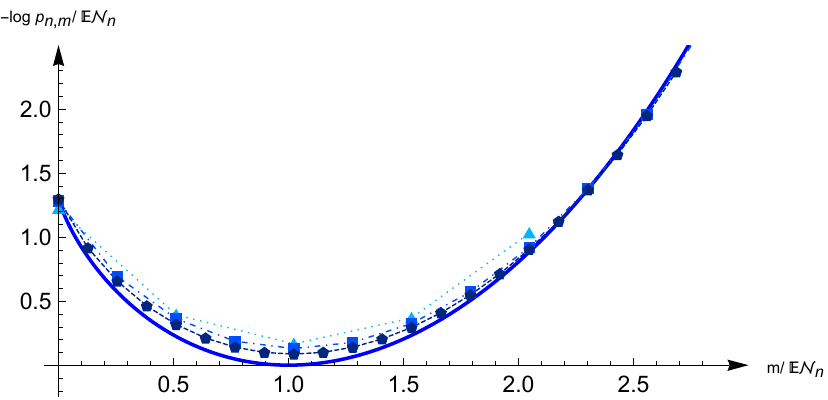}
        \end{center}
        \subcaption{ Strong }
    \end{subfigure}
    \qquad
    \begin{subfigure}{6.8cm}
        \begin{center}
            \includegraphics[width=6.8cm]{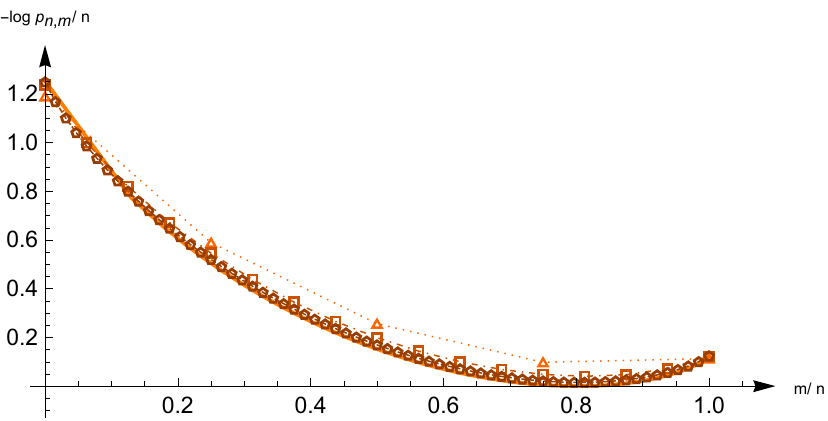}
        \end{center}
        \subcaption{ Weak  }
    \end{subfigure}
     \caption{Figures (a)--(c) show the numerical plots of $m \mapsto p_{n,m}$  in the strong asymmetry regime for $n = 8, 32, 128$, respectively. Figures (d)--(f) display the corresponding plots in the weak asymmetry regime.  Figures (g) and (h) present the graphs of $m \mapsto -\log p_{n,m}$ for both regimes (dotted line), together with their comparison to  $m \mapsto \mathbb{E}\mathcal{N}_n\,\phi_{\mathrm{s}}(m/\mathbb{E}\mathcal{N}_n)$ and $m \mapsto n\,\phi_{\mathrm{w}}(m/n)$ (full solid line).  
     Here, we set $\tau = 1/2$ for the strong asymmetry (blue) and $\alpha = 1$ for the weak asymmetry (orange). The formula~\eqref{eq for det formula of gen function} is used for the numerical evaluations. } \label{Fig_numerics of pnm}
\end{figure}

It is convenient to restate Theorem~\ref{Thm_main} for later discussions. For this, we write 
\begin{align} \label{eq. def Psi sA}
    \Psi_{\rm s}(z) & \equiv \Psi_{\rm s}(z;\tau) := -\sqrt{\frac{1+\tau}{1-\tau} \frac{1}{2\pi}}  \Li_{3/2}(1-z),
    \\
    \Psi_{\rm w}(z) & \equiv \Psi_{\rm w}(z; \alpha) := \frac{2}{\pi} \int_0^1 \log\Big( 1- (1-z)e^{-\alpha^2 s^2}\Big) \sqrt{1-s^2} \, ds \label{eq. def Psi wA}
\end{align}
and define 
\begin{align} \label{eq. def phi sA rate function}
 \phi_{\rm s}(x) \equiv \phi_{\rm s}(x;\tau):= \sup_{u \in \R} \Big\{ \frac{x}{2}u - \Psi_{\rm s}(e^u) \Big\},  \qquad  \phi_{\rm w}(x) \equiv\phi_{\rm w}(x;\alpha):=  \sup_{u \in \R} \Big\{ \frac{x}{2}u - \Psi_{\rm w}(e^u) \Big\}. 
\end{align}
Since the functions $\Psi_{\rm s}(e^u)$ and $\Psi_{\rm w}(e^u)$ are convex, the suprema in \eqref{eq. def phi sA rate function} are attained at the values of $u$ satisfying 
$$\frac{x}{2}= \frac{d}{du}\Psi_{\rm s}(e^u), \qquad \frac{x}{2}= \frac{d}{du}\Psi_{\rm w}(e^u),$$ 
respectively.
Therefore, it follows that
\begin{equation}
    \phi_{\rm s}(x) = -\Psi_{\rm s}(\Phi_{\rm s}^{-1}(\tfrac{x}{2})) + \tfrac{x}{2} \log \Phi_{\rm s}^{-1}(\tfrac{x}{2}), \qquad         \phi_{\rm w}(x) = - \Psi_{\rm w}(\Phi_{\rm w}^{-1}(\tfrac{x}{2})) + \tfrac{x}{2} \log \Phi_{\rm w}^{-1}(\tfrac{x}{2}) ,
\end{equation}
where 
\begin{align}
    \Phi_{\rm s}(z) \equiv \Phi_{\rm s}(z;\tau) &:= \frac{d}{du} \Psi_{\rm s}(e^u)\,\bigg|_{u=\log z} = \sqrt{\frac{1+\tau}{1-\tau} \frac{1}{2\pi}} \frac{z \Li_{1/2}(1-z)}{1-z}, \label{eq. def Phi sA}
    \\
 \label{eq. def Phi wA} 
        \Phi_{\rm w}(z) \equiv \Phi_{\rm w}(z;\alpha)  &:= \frac{d}{du} \Psi_{\rm w}(e^u)\,\bigg|_{u=\log z} = \frac{2}{\pi} \int_0^1 \frac{z \, e^{-\alpha^2 s^2}}{1- (1-z)e^{-\alpha^2 s^2}} \sqrt{1-s^2} \, ds.    
\end{align}
Now, Theorem~\ref{Thm_main} can be restated as follows: suppose $2m/\sqrt{2n} \to x \in (0,\infty)$ as $n \to \infty$ in the strong asymmetry regime, while $m/n \to x \in (0,1)$ as $n\to\infty$ in the weak asymmetry regime.
Then 
\begin{align}
  \lim_{n\to\infty} \frac{\log p_{2n,2m} }{\sqrt{2n}}   = -\phi_{\rm s}(x) & \quad \textup{for the strong asymmetry}, \label{eq. main restate SA}
  \\
    \lim_{n\to\infty} \frac{ \log p_{2n,2m} }{2n}   = - \phi_{\rm w}(x) & \quad \textup{for the weak asymmetry}. \label{eq. main restate WA}
\end{align}

The asymptotic behaviours of the rate functions $\phi_{\rm s}$ and $\phi_{\rm w}$ characterise how the deviation probabilities interpolate between the Gaussian fluctuation regime and the extreme large deviation regime, thereby clarifying the crossover structure between them. 
 The asymptotic behaviours of $\phi_{\rm s}$ as $x \to 0$ and $x \to \infty$ are given by \eqref{phi s 0} and \eqref{phi s infty}, respectively. Similarly, the asymptotic behaviours of $\phi_{\rm w}$ as $x \to 0$ and $x \to 1$ are given by \eqref{phi w x to 0} and \eqref{phi w x to 1}, respectively. These asymptotics provide natural matchings with the previous findings discussed in the previous section.

\subsection{Discussions} We now discuss several aspects and implications of Theorem~\ref{Thm_main}. 
 
\begin{rem}[Universal form in the strong asymmetry regime, cf. \cite{Fo25}]  
Notice that the right-hand side of \eqref{eqn for mod dev universal form} is independent of the choice of $\tau \in [0,1)$.  
This observation suggests a possible universality, which indeed extends beyond the present model.  
In a recent work by Forrester \cite[Proposition~3]{Fo25}, the analogous result for the spherical ensemble was established, where the same form as in \eqref{eqn for mod dev universal form} arises.  
In order to compare \eqref{eqn for mod dev universal form} and \cite[Proposition~3]{Fo25}, one needs the integral identity  
\begin{equation} \label{eq. integral rep of Polylog Li 3/2}
\Li_{3/2}(z)= -\frac{2}{\sqrt{\pi}} \int_0^\infty \log \Big( 1- z e^{-t^2} \Big) \,dt,
\end{equation}
which follows from the power series expansion of the logarithm on the right hand side for $|z|<1$, as well as the analytic continuation for $z \in \C \setminus [1,\infty)$.  
For the case $x = 0$ in \eqref{eqn for mod dev universal form}, corresponding to the extremal probability that there are no real eigenvalues, the possibility of such a universal form was discussed in \cite[Remark~1.4]{BMS25} and \cite[Section~3.3]{Fo25}. 

We conclude this remark on universality by noting that the same large deviation function $\phi_{\rm s}(x)$ in (\ref{eq. def phi sA rate function}) describes the current fluctuations in the simple symmetric exclusion process (SSEP), first computed in~\cite{DG09}, see Eq.~(2) there with $\rho_a = 0$ and $\rho_b = 1$. Note that a formula similar to Eq.~\eqref{eq_left tail strong}, specialized to the case $\tau=0$, has also appeared in very different contexts, such as in the nonequilibrium dynamics of the two-dimensional classical Ising model—specifically in the melting of a quadrant—and in the quenched dynamics of the quantum XXX chain \cite{KMS14,Ste17}. It is also interesting to note that a similar connection was found between the counting statistics in the complex Ginibre ensemble \cite{LGCCKMS19} and the current fluctuations for a Brownian gas of non-interacting particles. This points to further connections between Ginibre ensembles and the statistical mechanics particle systems, as previously noticed in \cite{GPTZ18} in the context of annihilating particle systems on the line.
\end{rem}

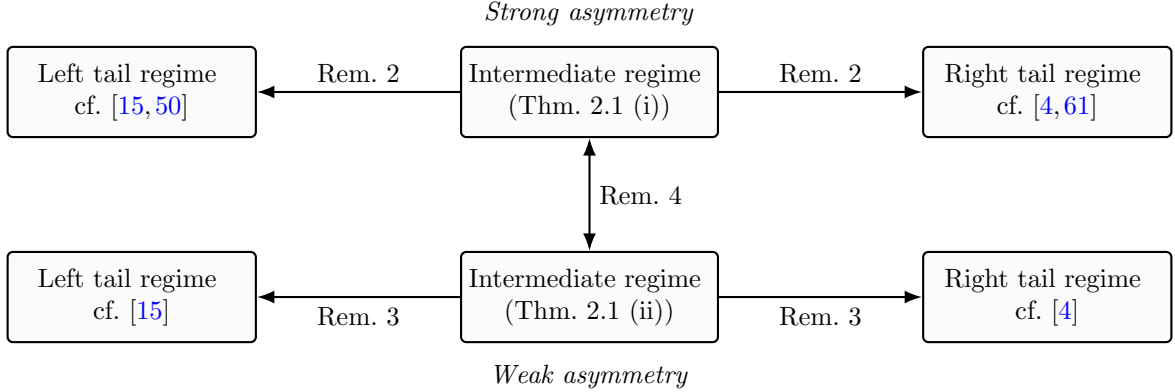
\begin{figure}
    \centering
    \begin{tikzpicture}[
  >=Latex, thick, 
  node distance = 15mm and 27mm,
  word/.style={
    draw=black,               
    rounded corners=2pt,   
    fill=gray!3,            
    minimum width=33mm,
    minimum height=12mm,
    inner sep=3pt,
    font=\normalsize,
    align=center, scale=0.89
  }
]

\node[word] (TL) {Left tail regime \smallskip \\ cf. \cite{KPTTZ15,BMS25}};
\node[word] (TC) [right=of TL] {Intermediate regime \smallskip \\ (Thm.~\ref{Thm_main} (i))};
\node[word] (TR) [right=of TC] {Right tail regime \smallskip \\ cf. \cite{MPTW16,ABL25}};

\node[word] (BL) [below=of TL] {Left tail regime \smallskip \\ cf. \cite{BMS25}};
\node[word] (BC) [below=of TC] {Intermediate regime \smallskip \\ (Thm.~\ref{Thm_main} (ii))};
\node[word] (BR) [below=of TR] {Right tail regime \smallskip \\ cf. \cite{ABL25}};
 
\draw[<->] (TC) -- node[midway,right] {Rem.~\ref{Rem_strong weak}} (BC);
 
\draw[->] (TC) -- node[midway,above] {Rem.~\ref{Rem_strong left right}} (TL);
\draw[->] (TC) -- node[midway,above] {Rem.~\ref{Rem_strong left right}} (TR);
 
\draw[->] (BC) -- node[midway,below] {Rem.~\ref{Rem_weak left right}} (BL);
\draw[->] (BC) -- node[midway,below] {Rem.~\ref{Rem_weak left right}} (BR);

\node[above=1.5mm of TC] {\textit{Strong asymmetry}};
\node[below=1.5mm of BC] {\textit{Weak asymmetry}};
\end{tikzpicture}
    \caption{The diagram illustrates the interrelations between our main results and the existing literature, the details of which can be found in Remarks~\ref{Rem_strong left right}, \ref{Rem_weak left right}, and~\ref{Rem_strong weak}.}
    \label{Fig_remarks on interrelations}
\end{figure}

Next, we discuss various limiting cases and interconnections between our main results and the existing literature.  
A summary is provided in Figure~\ref{Fig_remarks on interrelations}.

\begin{rem}[Extremal case; matching to the left and right tail large deviations in the strong asymmetry regime]  \label{Rem_strong left right}
Note that since $\Li_s(1)=\zeta(s)$, we have  
        \begin{equation} \label{phi s 0}
            \lim_{x\to 0+} \phi_{\rm s}(x) = - \Psi_{\rm s}(0) = \sqrt{\frac{1-\tau}{1+\tau} \frac{1}{2\pi}} \zeta(\tfrac{3}{2}). 
        \end{equation}  
Thus, heuristically, by taking the limit $x \to 0$ in our result \eqref{eq. main restate SA}, we can recover the previous finding \eqref{eq_left tail strong}.

On the other hand, the opposite limit $x \to \infty$ in \eqref{eq. main restate SA} is more delicate.
Clearly, one cannot expect to directly recover \eqref{eq_right tail strong}, due to the different scalings $O(\sqrt{n})$ and $O(n)$.
In the special case $\tau=0$, however, the Coulomb gas approach was used in \cite{MPTW16} to show that, for $\gamma \in (0,1)$,  
        \begin{equation*}
            \lim_{n\to\infty} \frac{\log p_{n, \gamma n}}{n^2} = - \mathcal{I}[\mu_\gamma],  
        \end{equation*}
   where $\mu_\gamma$ is the minimiser of the energy functional 
        \begin{equation*}
            \mathcal{I}[\mu] := \frac{1}{2} \bigg( \int |z|^2 \,d\mu(z) - \iint \log |z-w| \,d\mu(z) \, d\mu(w) \bigg) - \frac{3}{8} 
        \end{equation*}
among all probability measures that place mass $(1 - \gamma)$ on $\mathbb{C}\setminus \R$, and mass $\gamma$ on $\R$.
See also \cite{Fo25} for recent progress on this electrostatic approach; cf. \cite{BF25a}. 

Unlike in \cite{Fo25}, where the energy functional is explicitly evaluated for the spherical ensemble, no closed-form expression is available in the GinOE case. Consequently, a quantitative analysis of $\mathcal{I}[\mu_\gamma]$ remains challenging.  
Nevertheless, one may still expect a relation between the limit $x \to \infty$ in our formulation and the limit $\gamma \to 0$ in $\mathcal{I}[\mu_\gamma]$.  
Indeed, direct computations from \eqref{eq. def phi sA rate function} show that 
\begin{equation} \label{phi s infty}
\phi_{\rm s}(x) \sim  \frac{1-\tau}{1+\tau} \frac{\pi^2}{48} x^3, \qquad \textup{as }x \to \infty, 
\end{equation} 
where we used the asymptotic of the polylogarithm $\Li_{s}(\pm e^u) = - u^s/\Gamma(s+1) + o(u^s)$, as $\re u \to \infty$.
Although Theorem~\ref{Thm_main} is valid only when $x$ is kept fixed as $n\to\infty$, this heuristically suggests that
\begin{equation} \label{eq. log p asymp x to inf}
    \log p_{2n, 2m} \sim - \frac{1+\tau}{1-\tau} \frac{\pi^2}{48} x^3 \sqrt{2n} , \qquad \textup{as } \frac{2m}{\sqrt{2n}}= x \to \infty. 
\end{equation}
For the spherical ensemble, a similar cubic divergence of the rate function reproduces the $\gamma \to 0$ limit of $\mathcal{I}[\mu_\gamma]$, see \cite[Eq.~(3.12)]{Fo25}. We also mention that the opposite limit $\gamma \to 1$ should be read off from \eqref{eq_right tail strong}, leading to 
${\mathcal{I}}[\mu_{\gamma = 1}]=\frac14 \log ( \frac{2}{1+\tau}). $ 
\end{rem}

\begin{rem}[Extremal case; matching to the left and right tail large deviations in the weak asymmetry regime]   \label{Rem_weak left right}
In contrast to the strong asymmetry regime, in the weak asymmetry regime both extremal limits $x \to 0$ and $x \to 1$ in \eqref{eq. main restate WA} can be easily connected with previous findings.  
For the left-tail limit $x \to 0$, we have  
        \begin{equation} \label{phi w x to 0}
            \lim_{x \to 0+} \phi_{\rm w}(x) = -\Psi_{\rm w}(0) = -\frac{2}{\pi} \int_0^1 \log\Big( 1- e^{-\alpha^2 s^2}\Big) \sqrt{1-s^2} \, ds.
        \end{equation}
The right-hand side coincides with \eqref{eq_left tail weak}. By continuity, 
this provides stronger conjectural evidence that \eqref{eq_left tail weak} in 
fact holds as an equality.  
For the right-tail limit $x \to 1$, we obtain  
        \begin{align} \label{phi w x to 1}
            \lim_{x\to1-}\phi_{\rm w}(x)  =  \sup_{u \in \R} \Big\{ \frac{u}{2} - \Psi_{\rm w}(e^u) \Big\} = \frac{\alpha^2}{8},
        \end{align}
       which coincides with \eqref{eq_right tail weak}.  
The last equality follows from the observation that  
\begin{align*}
\frac{d}{du}\Big(\frac{u}{2} - \Psi_{\rm w}(e^u) \Big)  = \frac{1}{2} - \frac{2}{\pi} \int_0^1 \frac{e^{u-\alpha^2 s^2}}{e^{u - \alpha^2 s^2}+ 1 - e^{-\alpha^2 s^2}} \sqrt{1-s^2} \,ds
 > \frac{1}{2} - \frac{2}{\pi} \int_0^1 \sqrt{1-s^2} \,ds = 0,
\end{align*}
which leads to 
\begin{align*}
    \sup_{u \in \R} \Big\{ \frac{u}{2} - \Psi_{\rm w}(e^u) \Big\}  = \lim_{u \to \infty} \Big( \frac{u}{2} - \Psi_{\rm w}(e^u) \Big) =  \frac{u}{2} - \frac{2}{\pi} \int_0^1 (u - \alpha^2 s^2) \sqrt{1-s^2} \,ds = \frac{\alpha^2}{8}.
\end{align*}
We stress that Theorem~\ref{Thm_main} is valid only when $x$ is kept fixed as $n\to\infty$. Therefore, these computations do not constitute a rigorous proof of \eqref{eq_left tail weak} and \eqref{eq_right tail weak}.
\end{rem}


By definition, the rate functions in the moderate deviation probabilities are closely connected with the typical expected value. Specifically, the expected value is recovered from the minimiser of the rate function, while the second-order correction at this point, the curvature of the rate function at the critical point, is related to the variance. Using our closed formulas, we can verify these facts, which we formulate in the following proposition.

\begin{proposition}[\textbf{Minimiser and curvature of the rate functions}] \label{Prop_minimisers and curvature} Let $\phi_{\rm s}$ and $\phi_{\rm w}$ denote the rate functions in \eqref{eq. def phi sA rate function}. 
\begin{itemize}
    \item The minima of the rate functions $\phi_{\rm s}$ and $\phi_{\rm w}$ 
    occur at the limiting expected value of $\mathcal{N}_{2n}$ (normalised by 
    the typical scale). In other words, if $x_{\rm s}$ and $x_{\rm w}$ denote 
    the unique minimisers of $\phi_{\rm s}$ and $\phi_{\rm w}$, then  
 \begin{equation} \label{def of xs and xw}
        x_{\rm s}  = \sqrt{\frac{1+\tau}{1-\tau} \frac{2}{\pi}} = \lim_{n\to\infty} \frac{\mathbb{E}\mathcal{N}_{2n}}{\sqrt{2n}}, \qquad    x_{\rm w}   = c(\alpha) = \lim_{n\to \infty} \frac{\mathbb{E} \mathcal{N}_{2n}}{2n},
    \end{equation} 
    \item The curvatures of the rate functions $\phi_{\rm s}$ and $\phi_{\rm w}$ 
    at their minima are the reciprocal of the variance of $\mathcal{N}_{2n}$ 
    (normalised by the typical scale). More precisely,   
        \begin{align}
        \phi_{\rm s}''(x_{\rm s}) &= \frac{1}{(2-\sqrt{2}) \sqrt{\frac{1+\tau}{1-\tau} \frac{2}{\pi}} } = \lim_{n\to\infty}\frac{\sqrt{2n}}{\textup{Var} \mathcal{N}_{2n}}, 
        \\
           \phi_{\rm w}''(x_{\rm w})& =  \frac{1}{2 (c(\alpha) - c(\sqrt{2}\alpha) )} = \lim_{n \to \infty} \frac{2n}{\textup{Var} \mathcal{N}_{2n}}.
    \end{align}
\end{itemize}
\end{proposition}

This proposition will be shown in Subsection~\ref{Subsection_proof of Props}. 

\begin{rem}[Interpolating properties of the rate functions]   \label{Rem_strong weak}
A notable feature of the weak asymmetry regime is that it connects naturally with the strong asymmetry regime when the parameters are chosen appropriately. This relationship has been studied extensively in the literature (see, e.g. \cite[Section~2.1]{BKLL23} and \cite[Remark 1.4]{BMS25}), and it also holds in our present formulation.
Specifically, suppose $m/\E\mathcal{N}_{2n} \to y$ as $n \to \infty$.
By \eqref{eq. mean and variance sH}, \eqref{eq. mean and variance wH}, \eqref{def of xs and xw} and Theorem~\ref{Thm_main}, we have
\begin{equation} \label{eq. renormalised rate function}
    \lim_{n\to\infty}\frac{\log p_{2n, 2m}}{\E \mathcal{N}_{2n}} = \begin{dcases}
        -\frac{\phi_{\rm s}(x_{\rm s}y)}{ x_{\rm s} } & \quad \textup{for the strong asymmetry},
        \\
        -\frac{\phi_{\rm w}(x_{\rm w} y)}{ x_{\rm w} } & \quad \textup{for the weak asymmetry}.
    \end{dcases}
\end{equation}
Then, assuming $1 \ll \alpha \ll \sqrt{2n}$ as $n\to\infty$, the (renormalised) rate function in \eqref{eq. renormalised rate function} for the weak asymmetry coincides with that for the strong asymmetry in the limit $\alpha \to \infty$. Indeed, this is a consequence of the following observation.

For the strong asymmetry, it follows from straightforward computations using the definition \eqref{eq. def phi sA rate function} that
\begin{align} \label{eq. renormalised Psi Phi sA}
    \frac{\Psi_{\rm s}(x)}{ x_{\rm s} } = - \frac{\Li_{3/2}(1-x)}{2},
    \qquad
    \frac{\Phi_{\rm s}(x)}{ x_{\rm s} }  =  \frac{x \Li_{1/2}(1-x)}{2(1-x)}.
\end{align}
For the weak asymmetry, we suppose $1 \ll \alpha \ll \sqrt{2n}$ as $n \to \infty$. Then the leading contribution of the integral in \eqref{eq. def Psi wA} comes from small $s\ll 1$. Thus, we have
\begin{align*}
    \Psi_{\rm w}(x) = \frac{2}{\pi} \int_0^1 \log\Big( 1 - (1-x) e^{-\alpha^2 s^2}\Big)  \,ds  \sim  \frac{2}{\pi\alpha} \int_0^\infty \log\Big( 1 - (1-x) e^{-t^2}\Big)  \,dt  = - \frac{1}{\sqrt{\pi \alpha^2}} \Li_{3/2}(1-x) ,
\end{align*} 
as $\alpha \to \infty$.
Similar computations give rise to
\begin{align}
    \Phi_{\rm w}(x) \sim \frac{2}{\pi \alpha} \int_0^\infty \frac{x e^{-t^2}}{1 - (1-x) e^{-t^2}}  \,dt 
    = \frac{1}{\sqrt{\pi \alpha^2}} \frac{x \Li_{1/2}(1-x)}{1-x} .
\end{align}
On the other hand, using $I_\nu(x) = e^x/\sqrt{2\pi x} \,(1+o(1))$ as $x \to \infty$, we have $c(\alpha) = 2/\sqrt{\pi \alpha^2} \,(1+o(1))$ as $\alpha\to\infty$.
Hence, putting the above asymptotics together, we obtain
\begin{align} \label{eq. renormalised Psi Phi wA}
    \lim_{\alpha \to \infty}\frac{\Psi_{\rm w}(x)}{x_{\rm w}} = - \frac{\Li_{3/2}(1-x)}{2}, \qquad \lim_{\alpha \to \infty}\frac{\Phi_{\rm w}(x)}{x_{\rm w}} = \frac{x\Li_{1/2}(1-x)}{2(1-x)}.
\end{align}
These coincide with \eqref{eq. renormalised Psi Phi sA}, implying that the renormalised rate function in \eqref{eq. renormalised rate function} for the weak asymmetry regime matches that for the strong asymmetry regime in the limit $\alpha \to \infty$.
\end{rem}

Our next result concerns the asymptotic behaviour of the generating function of $p_{2n,2k}$ and the cumulants of $\mathcal{N}_{2n}$.  
The following proposition is indeed a key step in the proofs of our main results, while also being of independent interest.  
To state it, we recall that the Stirling number of the second kind $S(n,k)$ is defined by  
\begin{equation*}
S(n,k) = \frac{1}{k!} \sum_{j=0}^k (-1)^{k-j} \binom{k}{j} j^n,
\end{equation*}
see \cite[Section 26.8]{NIST}.
We also recall that for a random variable $X$, the $\ell$-th cumulant $\kappa_\ell(X)$ is defined by
\begin{equation*}
    \kappa_\ell(X) = \frac{d^\ell}{dt^\ell} \log \mathbb{E}[e^{t X}]\Big|_{t=0}.
\end{equation*} 

\begin{proposition}[\textbf{Asymptotic behaviour of the generating function and cumulants}]\label{Prop_generating function limit} $ $ 
\begin{itemize}
    \item[\textup {(i)}] \textbf{\textup{(Strong asymmetry regime)}}  Let $\tau \in [0,1)$ be fixed. Then we have 
    \begin{equation}  \label{eq. limit of gen func sH}
        \lim_{n\to\infty} \frac{1}{\sqrt{2n}} \log\Big( \sum_{k=0}^n z^k p_{2n,2k} \Big) = \Psi_{\rm s}(z), 
    \end{equation} 
   locally uniformly in $\C \setminus (-\infty, 0]$, where $\Psi_{\rm s}$ is given by \eqref{eq. def Psi sA}. 
    Furthermore for any $\ell \in \mathbb{N}_0$, we have 
    \begin{equation} \label{cumulant lims sH_v2}
        \lim_{n\to\infty} \frac{\kappa_\ell(\mathcal{N}_{2n}) }{\sqrt{2n}} = 2^\ell \sum_{m=1}^\ell (-1)^{m+1}  (m-1)! S(\ell,m) \sqrt{\frac{1+\tau}{1-\tau} \frac{1}{2\pi m}} . 
    \end{equation} 
    \item[\textup {(ii)}] \textbf{\textup{(Weak asymmetry regime)}}   Let $   \tau = 1 - \frac{\alpha^2}{2n}$ with fixed $\alpha \in (0,\infty)$. Then we have 
      \begin{equation} \label{eq. limit of gen func wH}
        \lim_{n\to\infty} \frac{1}{2n} \log\Big( \sum_{k=0}^n z^k p_{2n,2k} \Big) = \Psi_{\rm w}(z), 
    \end{equation}
    locally uniformly in $\C \setminus (-\infty, 0]$, where $\Psi_{\rm w}$ is given by \eqref{eq. def Psi wA}. 
    Furthermore for any $\ell \in \mathbb{N}_0$, we have  
    \begin{equation} \label{cumulant lims wH_v2}
        \lim_{n\to\infty} \frac{\kappa_\ell( \mathcal{N}_{2n} )}{2n} =  2^{\ell} \sum_{m=1}^\ell (-1)^{m+1} (m-1)! S(\ell,m) \frac{ c(\sqrt{m}\alpha) }{2},
    \end{equation}
    where $c(\alpha)$ is given by \eqref{eq. coeff avg number}.
\end{itemize}
\end{proposition}

This proposition will also be shown in Subsection~\ref{Subsection_proof of Props}. 
In particular, Proposition~\ref{Prop_generating function limit} (i) provides a slight extension of \cite[Theorem~1.5]{BMS25}.

Notice also that the first two cumulants, $\ell=1,2$ in \eqref{cumulant lims sH_v2} and \eqref{cumulant lims wH_v2} reproduce the mean and variance given in \eqref{eq. mean and variance sH} and \eqref{eq. mean and variance wH}. For $\ell=3$, we have
\begin{align*}
    \lim_{n \to \infty} \frac{\kappa_3(\mathcal{N}_{2n})}{\sqrt{2n}}  = \sqrt{\frac{1+\tau}{1-\tau} \frac{1}{\pi}}\, \frac{4}{3} \, (3\sqrt{2}-9 +2\sqrt{6})
\end{align*}
for the strong asymmetry and
\begin{equation*} 
        \lim_{n \to \infty} \frac{\kappa_3(\mathcal{N}_{2n})}{2n} = 4e^{-\frac{\alpha^2}{2}} \Big( I_0\big(\tfrac{\alpha^2}{2}\big) + I_1\big(\tfrac{\alpha^2}{2}\big) \Big)
        - 12 e^{-\alpha^2} \Big( I_0\big(\alpha^2\big) + I_1\big(\alpha^2\big) \Big)+ 8 e^{-\frac{3\alpha^2}{2}} \Big( I_0\big(\tfrac{3\alpha^2}{2}\big) + I_1\big(\tfrac{3\alpha^2}{2}\big) \Big) 
\end{equation*}
for the weak asymmetry.

\subsection{Idea of the proof}
We conclude this section by outlining heuristics for how Theorem~\ref{Thm_main} can be derived from Proposition~\ref{Prop_generating function limit}, together with the actual proof strategy that makes these heuristics rigorous. Indeed, the underlying heuristic argument captures main ingredients, but establishing it rigorously requires essential mathematical foundations, which will be addressed in Theorem~\ref{Thm. zero dist and exp profile}. 

We first note that the Laplace transform of $\mathcal{N}_{2n}$ is the 
generating function of $p_{2n,2k}$, namely  
\begin{align}\label{eq:Laplace}
    \E e^{s \mathcal{N}_{2n}} = \sum_{k=0}^n e^{2 s k } p_{2n,2k} = \sum_{k=0}^n z^k p_{2n,2k} \bigg|_{z = e^{2s}}. 
\end{align}
Note that $p_{2n,2k+1}=0$ for all $k=0,\ldots,n-1$, since the number of real eigenvalues must have the same parity as the matrix dimension.
The first key step in the proof of Theorem~\ref{Thm_main} is Proposition~\ref{Prop_generating function limit}, which provides  
\begin{equation} \label{gen function asymp Lap}
\sum_{k=0}^n e^{2 s k } p_{2n,2k} \sim \begin{cases}
\exp\Big[ \sqrt{2n} \, \Psi_{ \rm s }( e^{2s} ) \Big] &\textup{for the strong asymmetry},
\smallskip 
\\
\exp\Big[ 2n \, \Psi_{ \rm w }( e^{2s} ) \Big] &\textup{for the weak asymmetry},
\end{cases}  
\end{equation}
as $n \to \infty$.  
Since our interest is in $p_{2n,2k}$ for moderate deviation ranges of $k$, the natural question is: \smallskip 
\begin{center}
\emph{If we know the asymptotic behaviour of the polynomials with increasing degree ($\sum_{k=0}^n p_{2n,2k} z^k$), \\ how can we derive the asymptotics of their coefficients ($ p_{2n,2k}$)}?
\end{center} \smallskip  
To this end, we first assume the existence of a convex limiting \emph{exponential profile} for $p_{2n,2k}$; that is, there exists a convex function $\phi:[0,\infty)\to\R$ such that  
\begin{equation} \label{eq. Gregory's ansatz}
    p_{2n, 2k} \sim 
    \begin{cases}
       \exp\Big[ - \sqrt{2n} \, \phi_{ \rm s }\Big(\dfrac{2k}{\sqrt{2n}}\Big) \Big] &\textup{for the strong asymmetry},
\smallskip 
\\  
 \exp\Big[ - 2n \, \phi_{ \rm w }\Big(\dfrac{k}{n}\Big) \Big] & \textup{for the weak asymmetry}.
    \end{cases} 
\end{equation} 
From this ansatz and Laplace's method, $\Psi$ can be identified as the Legendre transform of $\phi$. More precisely, in the strong asymmetry case,  
\begin{align*}
    \E e^{s \mathcal{N}_{2n}} \sim \sum_{k=0}^n \exp\Big[\sqrt{2n}\Big(s \frac{2k}{\sqrt{2n}} - \phi\Big(\frac{2k}{\sqrt{2n}}\Big) \Big)\Big] \sim \exp\Big[ \sqrt{2n} \sup_{x\in [0,\infty)}\{s x - \phi(x)\} \Big]
\end{align*}
and an analogous expression holds in the weak asymmetry case. Therefore, 
by \eqref{gen function asymp Lap}, we have $\Psi(e^{2s}) = \phi^*(s)$, where $f^*$ 
denotes the Legendre transform of a function $f$. Since the Legendre 
transform is involutive for convex functions, we can in turn recover 
$\phi$ as the Legendre transform of $\Psi(e^{2s})$.

The main step required to make this argument rigorous is the verification of the ansatz \eqref{eq. Gregory's ansatz}. For this, we rely on a recent result of \cite{JKM25}. Roughly speaking, the existence of a limiting generating function implies the existence of a limiting exponential profile for the coefficients. Such a statement has been rigorously established when the leading asymptotics of the exponential profile match the order of the polynomial, namely $O(n)$. This is precisely the case in the weak asymmetry regime, where we can directly apply \cite[Theorem~5.1]{JKM25}. As a consequence, the convexity of the ansatz \eqref{eq. Gregory's ansatz} also follows.  

In contrast, in the strong asymmetry regime, the leading asymptotic is only $O(\sqrt{n})$. Hence, an extension of \cite[Theorem~5.1]{JKM25} is required to handle the case where only a small portion of the polynomial genuinely contributes to the leading-order asymptotic behaviour. This extension constitutes one of the main steps of our analysis and is formulated in Theorem~\ref{Thm. zero dist and exp profile}.

We refer the reader to Table~\ref{Table_proof ingred summary} for the summary of the key ingredients of the proofs presented in this paper. 

\begin{table}[!ht]
    \centering
    \begin{tabular}{  c  c  c  } 
        \hline
        \cellcolor{gray!0}  & \cellcolor{gray!0} \parbox[c]{3.5cm}{\centering  Strong Asymmetry }  & \cellcolor{gray!0} \parbox[c]{3.5cm}{\centering  Weak Asymmetry} \\
        \hline
        \cellcolor{gray!0} \parbox[c]{3.5cm}{\centering  Limit of \\ generating function} 
        & Proposition~\ref{Prop_generating function limit} (i) 
        & Proposition~\ref{Prop_generating function limit} (ii) \\
        \hline
        \cellcolor{gray!0} \parbox[c]{3.5cm}{\centering  Zeros and exponential profiles of polynomials}   & Theorem \ref{Thm. zero dist and exp profile} & \cite[Theorem 5.1]{JKM25} \\
        \hline
    \end{tabular}
    \caption{Strategy and main ingredients of the proofs presented in this paper.}
    \label{Table_proof ingred summary}
\end{table}

\section{The generating function of the number of real eigenvalues}

In this section, we compile some preliminaries and show Propositions~\ref{Prop_minimisers and curvature} and \ref{Prop_generating function limit}. 

\subsection{Preliminaries on the generating functions}
In this subsection, we collect known results on the generating function of $p_{2n,2k}$. For the case $\tau=0$, these results can be found in \cite{KPTTZ15}, while the general case $\tau \in [0,1)$ is treated in \cite{BMS25}.

The first main observation regarding the generating function arises from the underlying algebraic structure of the eGinOE, namely the fact that it forms a Pfaffian point process \cite{FN07,FN08}. Moreover, in even dimensions, the Pfaffian reduces to a determinant, leading to the following formula:
  \begin{align} \label{eq for det formula of gen function}
        \sum_{k=0}^n z^k  p_{2n,2k} = \det\Big[ {\rm I}_n + (z-1) M_n\Big],
    \end{align}
where ${\rm I}_n$ is the $n\times n$ identity matrix and $M_n$ is an $n\times n$ symmetric matrix with entries
    \begin{align*}
    \begin{split}
        [ M_n ]_{j,k} 
        &= \frac{1}{\sqrt{2\pi}} \frac{(\tau/2)^{j+k-2}}{\sqrt{\Gamma(2j-1)\Gamma(2k-1)}} \int_\R e^{-\frac{x^2}{1+\tau}} H_{2j-2}\Big(\frac{x}{\sqrt{2\tau}}\Big)  H_{2k-2}\Big(\frac{x}{\sqrt{2\tau}}\Big) \,d x.
    \end{split}
    \end{align*}
Here, $H_k$ is the $k$-th Hermite polynomial. 
Furthermore, the spectrum of $M_n$ is contained in the interval $(0,1)$.
These results were established in \cite[Proposition~2.1 and Lemma~2.3]{BMS25}. For the GinOE case, analogous determinantal formulas can also be found in \cite{AK07,KPTTZ15,Fo15a}.

\begin{lemma} \label{lem. trace power expansion of the generating function}
    For any $z \in \mathbb{C}$ with $ |z-1| < 1$, we have
    \begin{align}
        \log\Big( \sum_{k=0}^n z^k p_{2n,2k} \Big) = -\sum_{k = 1}^\infty \frac{1}{k} (1-z)^k \Tr  (M_n^k). \label{eq. trace power expansion of the generating function}
    \end{align} 
\end{lemma}
\begin{proof}
    Using \eqref{eq for det formula of gen function} and the matrix identity $\log \det A = \Tr \log A$, we have
    \begin{align} \label{eq. generating function as log-det and trace-log}
        \log\Big( \sum_{k=0}^n z^k p_{2n,2k} \Big) &= \log \det\Big[ {\rm I}_n + (z-1) M_n \Big] = \Tr \log\Big[ {\rm I}_n + (z-1) M_n\Big]. 
    \end{align}
    Since all the eigenvalues of $M_n$ lie in the interval $(0,1)$, for given $n$ and $|z-1|<1$, we have the Taylor expansion of matrix-valued logarithm
    \begin{equation*}
        \log\Big[ {\rm I}_n + (z-1) M_n\Big] =  \sum_{k=1}^\infty (-1)^{k+1} \frac{1}{k} (z-1)^k  M_n^k.
    \end{equation*}
    This gives rise to the desired result \eqref{eq. trace power expansion of the generating function}.
\end{proof}

The main step in \cite{BMS25} for the asymptotic analysis is the study of the moments $\Tr(M_n^k)$. In particular, it was shown in \cite[Propositions~2.5 and~2.6]{BMS25} that for any fixed $k \in \mathbb{N}$, the following limits hold:
\begin{itemize}
    \item For the strong asymmetry regime, 
    \begin{align} \label{eq. trace power limit sH}
        \lim_{n\to \infty} \frac{1}{\sqrt{2n}} \Tr  (M_n^k) = \sqrt{\frac{1+\tau}{1-\tau} \frac{1}{2\pi k}}.
    \end{align}
    \item For the weak asymmetry regime, 
     \begin{align} \label{eq. trace power limit wH}
        \lim_{n\to \infty} \frac{1}{2n} \Tr  (M_n^k) = \frac{e^{-k \alpha^2/2}}{2} \Big[ I_0\Big(\frac{k\alpha^2}{2} \Big) + I_1\Big(\frac{k\alpha^2}{2} \Big) \Big].
    \end{align}
\end{itemize}
These asymptotic behaviours play a central role in the subsequent analysis.

\subsection{Proof of Propositions~\ref{Prop_minimisers and curvature} and ~\ref{Prop_generating function limit}} \label{Subsection_proof of Props}

In this subsection, we show Propositions~~\ref{Prop_minimisers and curvature} and \ref{Prop_generating function limit}.

\begin{proof}[Proof of Proposition~\ref{Prop_minimisers and curvature}]
    
    Let us consider the rate function $\phi_{\rm s}$ given in \eqref{eq. def phi sA rate function} for a constant $\tau \in [0,1)$.
    Defining $\psi_{\rm s}(x) := \Psi_{\rm s}(e^x)$, we write $\phi_{\rm s} = \psi_{\rm s}^*$, where $f^*$ denotes the Legendre transform of a function $f$.
    It follows from direct computations that
    \begin{align}
        \psi_{\rm s}'(x) &= \sqrt{\frac{1+\tau}{1-\tau} \frac{1}{2\pi}} \frac{e^x}{1-e^x} \Li_{1/2}(1-e^x),
        \\
        \psi_{\rm s}''(x) &= \sqrt{\frac{1+\tau}{1-\tau} \frac{1}{2\pi}} \frac{e^x}{(1-e^x)^2} \Big[ \Li_{1/2}(1-e^x) - e^x \Li_{-1/2}(1-e^x)\Big].\label{eq:psi''}
    \end{align} 
    Given a twice-differentiable function $f$, we have identities
    \begin{equation*}
        (f^*)'(x) = (f')^{-1}(x), \qquad (f^*)''(x) = \frac{1}{f''((f')^{-1}(x))}.
    \end{equation*}
    Let us denote by $x_{\rm s}$ the minima of $\phi_{\rm s}$.
    Then, we have $0=\phi_{\rm s}'(x_{\rm s}) = \frac{1}{2} (\psi_{\rm s}^*)'(\frac{x_{\rm s}}{2}) = \frac{1}{2} (\psi_{\rm s}')^{-1}(\frac{x_{\rm s}}{2})$.
    Therefore,
    \begin{equation*}
        x_{\rm s} = 2 \psi_{\rm s}'(0) = \sqrt{\frac{1+\tau}{1-\tau} \frac{2}{\pi}} = \lim_{n\to\infty} \frac{\mathbb{E}\mathcal{N}_{2n}}{\sqrt{2n}},
    \end{equation*} 
    where we used \eqref{eq. mean and variance sH}, and similarly
    \begin{align*}
        \phi_{\rm s}''(x_{\rm s}) = \frac{1}{4} (\psi_{\rm s}^*)''(\frac{x_{\rm s}}{2}) =   \frac{1}{4 \psi_{\rm s}''((\psi_{\rm s}')^{-1}(\frac{x_{\rm s}}{2}))} = \frac{1}{4\psi_{\rm s}''(0)} = \frac{1}{(2-\sqrt{2}) \sqrt{\frac{1+\tau}{1-\tau} \frac{2}{\pi}} } = \lim_{n\to\infty}\frac{\sqrt{2n}}{\textup{Var} \mathcal{N}_{2n}}.
    \end{align*}

    Similarly, let us denote by $x_{\rm w}$ the minima of the rate function $\phi_{\rm w}$ and define $\psi_{\rm w}(x) := \Psi_{\rm w}(e^x)$.
    Then, we have
    \begin{align}
        \psi_{\rm w}'(x) &= \frac{2}{\pi} \int_0^1 \frac{e^x}{e^x + e^{\alpha^2 s^2}-1} \sqrt{1-s^2} \, ds,
        \\
        \psi_{\rm w}''(x) &=  \frac{2}{\pi} \int_0^1 \frac{e^x(e^{\alpha^2 s^2}-1)}{(e^x + e^{\alpha^2 s^2}-1)^2} \sqrt{1-s^2} \, ds.\label{eq:psi_s''}
    \end{align}
    Then, using $0=\phi_{\rm w}'(x_{\rm w}) =  \frac{1}{2} (\psi_{\rm w}^*)'(\frac{x_{\rm w}}{2}) = \frac{1}{2} (\psi_{\rm w}')^{-1}(\frac{x_{\rm w}}{2})$, we obtain
    \begin{equation*}
        x_{\rm w} = 2 \psi_{\rm w}'(0) = c(\alpha) = \lim_{n\to \infty} \frac{\mathbb{E} \mathcal{N}_{2n}}{2n},
    \end{equation*}
    and
    \begin{equation*}
        \phi_{\rm w}''(x_{\rm w}) = \frac{1}{4} \frac{1}{\psi_{\rm w}''(0)} = \frac{1}{2 (c(\alpha) - c(\sqrt{2}\alpha) )} = \lim_{n \to \infty} \frac{2n}{\textup{Var} \mathcal{N}_{2n}},
    \end{equation*}
    where the expectation and variance are given in \eqref{eq. mean and variance wH}.
    This completes the proof. 
\end{proof}

\begin{proof}[Proof of Proposition~\ref{Prop_generating function limit}] 
    We first prove the locally uniform convergence \eqref{eq. limit of gen func sH} for $z \in \C \setminus (-\infty,0]$.
    Suppose $|z+1|>\epsilon$ for some $\epsilon>0$.
    We observe from \eqref{eq. generating function as log-det and trace-log} that
    \begin{equation} \label{eq. gen func bound by trace}
        \log\Big( \sum_{k=0}^n z^k p_{2n,2k} \Big) = \Tr \log\Big[{\rm I}_n + (z-1) M_n\Big] 
        \leq c_\epsilon |z-1| \Tr M_n,
    \end{equation}
    where $c_\epsilon>0$ is a constant such that $|\log(1+w)| \leq c_\epsilon |w|$ for any $w \in \C \setminus(-\infty,0]$ satisfying $|w+1|>\epsilon$.
    Here, we also used the fact that $M_n$ is positive-definite.
    Together with \eqref{eq. trace power limit sH}, the above bound implies
    \begin{equation*}
        \bigg\{ \frac{1}{\sqrt{2n}} \log\Big( \sum_{k=0}^n z^k p_{2n,2k} \Big) \bigg\}_{n=1}^\infty
    \end{equation*}
    is a sequence of locally bounded holomorphic functions on $\C\setminus(-\infty,0]$. Hence, by Montel’s theorem, it is a normal family.
    On the other hand, by \cite[Theorem~1.5]{BMS25}, the sequence converges to $\Psi_{\rm s}$ on the disk $\{z \in \C: |z-1|<1\}$.
    Thus, applying the identity theorem, the assertion \eqref{eq. limit of gen func sH} follows. 
    
    Recall that by definition, for a random variable $X$ we have
    \begin{equation} \label{eq. cumulant and moment gen. func.}
        \log \E e^{t X} = \sum_{\ell=1}^\infty \frac{t^\ell}{\ell!} \kappa_\ell(X),
    \end{equation}
    where $\kappa_\ell(X)$ is the $\ell$-th cumulant of $X$.
    Then, \eqref{eq. limit of gen func sH} in turn implies that
    \begin{equation*}
        \log \E e^{t \mathcal{N}_{2n}} = - \sqrt{2n} \sqrt{\frac{1+\tau}{1-\tau} \frac{1}{2\pi}} \Li_{3/2}(1-e^{2t}) + o(\sqrt{n}),
    \end{equation*}
    as $n\to\infty$.
    Substituting the series expansion of the polylogarithm
    \begin{equation*}
        \Li_{3/2}(1-e^{2t}) = \sum_{\ell=1}^\infty \frac{(2t)^\ell}{\ell!} \sum_{m=1}^\ell \frac{(-1)^m}{\sqrt{m}} (m-1)! S(\ell,m),
    \end{equation*}
    which can be obtained along the lines presented in the supplementary material of \cite{PS18}. 
    Comparing the coefficients using \eqref{eq. cumulant and moment gen. func.}, we obtain the desired result \eqref{cumulant lims sH_v2}.

Next, we prove the second claim (ii). Since the eigenvalues of $M_n$ lie in the interval $(0,1)$, the power series \eqref{eq. trace power expansion of the generating function} is absolutely convergent for sufficiently large $n$ whenever $|z-1|<1$.
For $z\in \C$, use of Montel's theorem with \eqref{eq. trace power limit wH} and \eqref{eq. gen func bound by trace} shows that a sequence
\begin{equation*}
    \bigg\{ \frac{1}{2n} \log\Big( \sum_{k=0}^n z^k p_{2n,2k} \Big) \bigg\}_{n=1}^\infty
\end{equation*}
is locally uniformly convergent for any $z\in\C\setminus(-\infty,0]$.
Hence it suffices to establish \eqref{eq. limit of gen func wH} for $z \in (0,1)$, and the general case then follows from the identity theorem. 

    We compute the upper bound for the limiting generating function in \eqref{eq. limit of gen func wH}.
    Assume $z \in (0,1)$.
    Since $(1-z)^k \Tr M_n^k >0$, we deduce from \eqref{eq. trace power expansion of the generating function} that  for any $k_0\in\mathbb{N}$, 
    \begin{align*}
        \log\Big( \sum_{k=0}^n z^k p_{2n,2k} \Big) < -\sum_{k = 1}^{k_0} \frac{1}{k} (1-z)^k \Tr  M_n^k.
    \end{align*}
   Thus, together with \eqref{eq. trace power limit wH}, we have
    \begin{align*}
        \limsup_{n\to\infty} \frac{1}{2n} \log\Big( \sum_{k=0}^n z^k p_{2n,2k} \Big) &\leq -\lim_{n\to\infty} \sum_{k = 1}^{k_0} \frac{1}{k} (1-z)^k \frac{1}{2n} \Tr  M_n^k 
        \\
        &= -\sum_{k = 1}^{k_0} \frac{1}{2k} (1-z)^k e^{-k\alpha^2/2} \Big[I_0\Big( \frac{k\alpha^2}{2}\Big) + I_1\Big( \frac{k\alpha^2}{2}\Big)\Big].
    \end{align*}
    Letting $k_0 \to \infty$, we obtain 
    \begin{align}
    \begin{split} 
        &\quad \limsup_{n\to\infty} \frac{1}{2n} \log\Big( \sum_{k=0}^n z^k p_{2n,2k} \Big) \leq - \sum_{k = 1}^\infty \frac{1}{2k} (1-z)^k e^{-k\alpha^2/2} \Big[I_0\Big( \frac{k\alpha^2}{2}\Big) + I_1\Big( \frac{k\alpha^2}{2}\Big)\Big]
        \\
        &= - \frac{2}{\pi} \int_0^1 \sum_{k = 1}^\infty \frac{1}{k} (1-z)^k e^{-k\alpha^2 s^2} \sqrt{1-s^2} \,ds
        =\frac{2}{\pi} \int_0^1 \log\Big( 1- (1-z)e^{-\alpha^2 s^2}\Big) \sqrt{1-s^2} \, ds. \label{eq. limsup gen func wH}
    \end{split}
    \end{align}
    For the first equality, we use the integral identity (see e.g. \cite[Proof of Lemma A.1]{BMS25}) 
    \begin{equation} \label{int ideneity for c(x)}
        \int_0^1 e^{-x s^2} \sqrt{1-s^2} ds = \frac{\pi}{4} e^{-\frac{x}{2}} \Big[I_0\Big( \frac{x}{2}\Big) + I_1\Big( \frac{x}{2}\Big)\Big], 
    \end{equation}
    while for the second equality, we use the Taylor expansion of $\log(1-x)$.

    We show the lower bound for the limiting generating function in \eqref{eq. limit of gen func wH}.
    Let 
    \begin{align*}
        \log(1-x) = - \sum_{j=1}^{k} \frac{1}{j} x^j + R_k(x), \qquad R_{k}(x) := -\int_0^x \Big(\frac{x-t}{1-t}\Big)^{k} dt .
    \end{align*}
    Notice that for $x \in [0,1)$, we have $R_k(x) \ge - x^{k+1}/(1-x).$
    Furthermore, since the spectrum of $M_n$ is contained in the interval $(0,1)$, we have $\Tr M_n^k < n$ for any $k \in \mathbb{N}$.
    Therefore, for any $k_0\in\mathbb{N}$, we have
    \begin{align*}
        &\quad \frac{1}{2n} \log\Big( \sum_{k=0}^n z^k p_{2n,2k} \Big) > - \sum_{k = 1}^{k_0} \frac{ (1-z)^k  }{k}  \frac{1}{2n} \Tr  M_n^k - \sum_{k = k_0+1}^{\infty} \frac{  (1-z)^k }{2k} 
        \\
        &= -\sum_{k = 1}^{k_0} \frac{1}{k} (1-z)^k \frac{1}{2n} \Tr  M_n^k +  \frac{R_{k_0}(1-z)}{2}
        \geq -\sum_{k = 1}^{k_0} \frac{1}{k} (1-z)^k \frac{1}{2n} \Tr  M_n^k -  \frac{(1-z)^{k_0+1}}{2z}. 
    \end{align*} 
    Taking the limit $n\to\infty$, this yields
    \begin{align*}
        \liminf_{n\to\infty}\frac{1}{2n} \log\Big( \sum_{k=0}^n z^k p_{2n,2k} \Big) &\geq - \sum_{k = 1}^{k_0} \frac{1}{2k} (1-z)^k c(\sqrt{k}\alpha) - \frac{(1-z)^{k_0+1}}{2z}.
    \end{align*}
    Since $z \in (0,1)$, by taking limit $k_0\to \infty$, it follows from \eqref{eq. coeff avg number}, \eqref{int ideneity for c(x)} and the dominated convergence theorem that
    \begin{align}
        \liminf_{n\to\infty} \frac{1}{2n} \log\Big( \sum_{k=0}^n z^k p_{2n,2k} \Big) \geq \frac{2}{\pi} \int_0^1 \log\Big( 1- (1-z)e^{-\alpha^2 s^2}\Big) \sqrt{1-s^2} \, ds. \label{eq. liminf gen func wH}
    \end{align}
    Combining \eqref{eq. limsup gen func wH} and \eqref{eq. liminf gen func wH}, we obtain \eqref{eq. limit of gen func wH} for $z \in (0,1)$. This completes the proof of  \eqref{eq. limit of gen func wH}.

 Finally, we show \eqref{cumulant lims wH_v2}. 
  We have
    \begin{equation*}
        \log \E e^{t \mathcal{N}_{2n}} = \frac{4n}{\pi} \int_0^1 \log\Big( 1- (1-e^{2t})e^{-\alpha^2 s^2}\Big) \sqrt{1-s^2} \, ds + o(n),
    \end{equation*}
    as $n\to\infty$.
    By series expansions, we have
    \begin{equation*}
        \log\Big(1-(1-e^{2t}) e^{-\alpha^2 s^2}\Big) = \sum_{\ell=1}^\infty \frac{t^\ell}{\ell!} 2^\ell \sum_{m=1}^\ell (-1)^{m-1} (m-1)! S(\ell, m) e^{-m\alpha^2s^2}.
    \end{equation*}
    Thus, using Fubini's theorem and \eqref{int ideneity for c(x)}, we have
    \begin{equation*}
        \log \E e^{t \mathcal{N}_{2n}} = n\sum_{\ell=1}^\infty \frac{t^\ell}{\ell!} 2^\ell \sum_{m=1}^\ell (-1)^{m-1} (m-1)! S(\ell, m) e^{-m\alpha^2/2}\Big[ I_0\Big(\frac{m\alpha^2}{2}\Big) + I_1\Big(\frac{m\alpha^2}{2}\Big) \Big] + o(n),
    \end{equation*}
    as $n\to\infty$.
    Then, the conclusion follows by a comparison of coefficients using \eqref{eq. cumulant and moment gen. func.}.
\end{proof}

\section{Exponential profiles of the generating functions} \label{sec. Exponential profiles}

In this section, we will make the ansatz \eqref{eq. Gregory's ansatz} rigorous and complete the proof of Theorem~\ref{Thm_main}.

\subsection{A local Large Deviation Principle}
By the G\"artner-Ellis Theorem, convergence of the log-moment generating function as in Proposition \ref{Prop_generating function limit} implies a Large Deviation Principle (LDP) with speed $n$ or $\sqrt n$, respectively (see e.g.  \cite{DZ10}).  
From this viewpoint, \eqref{eq. Gregory's ansatz} corresponds to the following local LDP that holds under a convexity assumption of the limit log-Laplace transform $x\mapsto\Psi_{\rm s/w}(e^{x})$.

The following theorem gives a rigorous justification of the ansatz \eqref{eq. Gregory's ansatz}.

\begin{theorem}[\textbf{Exponential profile of coefficients}] \label{Thm. zero dist and exp profile}
	Consider a sequence of polynomials $\{P_{n}(z)\}_{n\in\mathbb{N}}$ of degree $n$ with non-positive real roots and $P_n(1)=1$, so that we can denote it by
	\begin{align}\label{eq:polynomials}
		P_n(z) = \sum_{k=0}^n a_{n,k} z^k= \prod_{j=1}^n \frac{z+\lambda_{n,j}}{1+ \lambda_{n,j}}, \qquad \lambda_{n,1},\dots,\lambda_{n,n} \ge 0.
	\end{align}
  Suppose there is a speed $c_n \to \infty$ as $n\to\infty$, such that
  \begin{align}
\Psi(z):=\lim_{ n \to \infty } \frac 1 {c_n}\log P_n(z),\quad z\in (0,\infty)
  \end{align}
 exists (which is automatically smooth) and $\Phi(z):=z\Psi'(z)$ is strictly increasing. 
 Define $\underline{m} := \lim_{x\to0+} \Phi(x)$, $\overline{m} := \lim_{x \to \infty} \Phi(x)$ and the exponential profile
 \begin{equation}
     g(x) := \Psi(\Phi^{-1}(x)) - x \log \Phi^{-1}(x)=-\sup_{t\in\R}\{tx-\Psi(e^t)\},
 \end{equation}
 which is a smooth, strictly concave function on $(\underline m,\overline m)$. Then, for any $[a,b]\subset(\underline{m},\overline{m})$, we have
    \begin{equation} \label{eq. exp profile of coeff}
        \sup_{k \in \{0,\dots,n\} \cap [ac_n,bc_n] } 
        \bigg| \frac{\log a_{n,k}}{c_n} - g\Big(\frac{k}{c_n}\Big) \bigg| \to 0, \qquad \textup{as } n \to \infty. 
    \end{equation}
\end{theorem}

 \begin{rem} $ $
     \begin{enumerate}
         \item Note that $\Phi$ is strictly increasing, if and only if $\Psi(e^x)$ is convex. By the Nevanlinna-Pick representation (see e.g. Theorem 10 in \cite[Chapter 3]{MS17}), $\Phi$ is the Stieltjes transform of some (not probability) measure $\mu_\infty$ with mass $\overline m$ on $\R$. It is the vague limit of the empirical zero measure 
\begin{align}\label{eq:empirical_distribution}
        	\mu_n := \frac{1}{c_n} \sum_{j=1}^n \delta_{-\lambda_{n,j}}\longrightarrow \mu_\infty\quad \text{as }n\to\infty.
        \end{align}   
        Hence, strict monotonicity of $\Phi$ is equivalent to $\mu_\infty\neq c\delta_0$. Similarly, $\Psi$ can be viewed as the logarithmic potential of $\mu_\infty$. 
        \smallskip 
        \item It is implicit that $c_n=O( n)$. Indeed, if $c_n\gg n$, then 
        $\Psi'(z)=0$ 
        for any $z>0$. Thus strict monotonicity of $\Phi$ cannot be achieved. 
    \end{enumerate}
 \end{rem}

 \begin{rem}
    The proof of Theorem \ref{Thm. zero dist and exp profile} follows the approach of \cite[Proof of Theorem 5.1]{JKM25}, wherein the real rootedness of $P_n$ is crucial in order to represent $P_n$ as the generating function of sums of Bernoulli random variables. Let us comment on how Theorem \ref{Thm. zero dist and exp profile} may be of independent interest as a generalization of the main tool in \cite{JKM25,JKM25II}. Therein, it is shown that the following two statements are equivalent:
    \begin{enumerate}
    	\item[(a)] The coefficients $a_{n,k}$ of $P_n$ have an exponential profile $g$, in the sense that $\frac 1 n \log a_{n,k} -g(k/n)\to 0$ as $n\to\infty $, uniformly in $(\underline m+\varepsilon)n\le k\le (\overline m -\varepsilon )n$ as in \eqref{eq. exp profile of coeff}.
        \smallskip 
    	\item[(b)] The empirical zero distribution of the polynomial $P_n$ converges weakly to some distribution $\mu$ on $[-\infty, 0]$, such that $\Phi(z)= \int_{(-\infty,0]} \frac{z}{z-y} \mu(dy)$ for $z>0$ 
    	is the inverse of $\alpha\mapsto e^{-g'(\alpha)}$. 
    \end{enumerate}
This corresponds to $c_n=n$ in Theorem \ref{Thm. zero dist and exp profile}, which extends to the case where also a negligible part $c_n=o(n)$ of the coefficients may be considered with a different scaling. This corresponds to zooming into the empirical distribution around 0, while forgetting the majority of zeros that drift away to $-\infty$. Indeed, if the above $\mu$ is not $\delta_{-\infty}$, then the conditions of Theorem \ref{Thm. zero dist and exp profile} fail for slow speeds $c_n=o(n)$. Thus, we can obtain moderate deviation asymptotics for these coefficients, \emph{because} the zero distribution of $P_n$ drifts to $-\infty$ reasonably fast.

Note that this is the case for the characteristic polynomial \eqref{eq for det formula of gen function}, whose zeros are $1-1/\lambda_j$ for $\lambda_j\in(0,1)$ being the eigenvalues of $M_n$. Thus, it is crucial that most eigenvalues of $M_n$ are close to $0$.
 \end{rem}

\begin{proof}[Proof of Theorem \ref{Thm. zero dist and exp profile}]
Let us first show that $\Psi$ is holomorphic on $D=\{z\in\C:\re z>0\}$. 
Since $P_n$ has only non-positive real roots and using the empirical measure \eqref{eq:empirical_distribution}, we write
    \begin{align*}
        \Psi_n(z) := \frac{1}{c_n} \log P_n(z) 
        = \int_{-\infty}^0 \log\Big( \frac{z-y}{1-y} \Big) \mu_n(dy) \to \Psi(z)
    \end{align*}
    for any $z \in (0,\infty)$, as $n \to \infty$ and where $\Psi_n$ is holomorphic on $D$. Note that the convergence $\Psi_n \to \Psi$ of convergent monotonic functions $\Psi_n(z)$ is locally uniform in $(0,\infty)$, but also
\begin{align*}
        \Psi'_n(z) =\frac{P_n'(z)}{c_nP_n(z)}= \int_{-\infty}^0 \frac{1}{z-y} \mu_n(dy) 
    \end{align*}
    are positive and strictly decreasing in $z\in (0,\infty)$. Furthermore, since $\mu_n$ is supported on the negative real axis, we have $|\Psi_n'(z)| \leq \Psi_n'(\re z)$ for all $z \in D$. Hence for any $z\in D_\delta:=\{z\in\C:\re z\ge \delta\}$, $\delta>0$, we have a uniform bound
    \begin{align*}
    |\Psi_n'(z)|\le \Psi_n'(\delta)\le \frac{1}{\delta/2}\int_{\delta/2}^{\delta}\Psi'_n(z)dz\to \frac 2 \delta (\Psi(\delta)-\Psi(\delta/2) ) <\infty.
    \end{align*}
This implies a locally uniform bound for $\Psi_n$ on compact subsets of $D_\delta$ by
    \begin{align*}
    |\Psi_n(z)|\le |\Psi_n(\delta)|+ \int_\delta^z |\Psi'_n(z)|dz <c_\delta(1+|z-\delta|)<\infty
    \end{align*}
for some constant $c_\delta>0$.
Therefore, Montel's theorem (or, Vitali-Porter)  implies that $\Psi_n$ converges locally uniformly to a holomorphic limit on $D_\delta$ that coincides with $\Psi$ on $(\delta,\infty)$ for all $\delta>0$, hence we call it $\Psi$.
    Similarly, we have locally uniform convergence of convergent (completely) monotonic functions: for $z \in (0,\infty)$, 
    \begin{align*}
        \Phi_n(z) :=  z\Psi'_n(z)=\frac z {c_n}\frac{P'_n(z)}{P_n(z)}=\int_{-\infty}^0 \frac{z}{z-y} \mu_n(dy) \to \Phi(z):=z\Psi'(z), 
    \end{align*}
    as $n \to \infty.$ 
    By strict monotonicity and continuity of $\Phi$, the inverse $\Phi^{-1}: (\underline{m}, \overline{m}) \to (0,\infty)$ is well-defined and continuous. Again by locally uniform convergence of convergent monotonic functions, the convergence $\Phi_n^{-1} \to \Phi^{-1}$ is locally uniform in $(\underline{m}, \overline{m})$. Since $-g$ is the Legendre transform of the strictly convex function $\Psi(e^x)$ on $x\in\R$, it follows that $g$ is automatically strictly concave and smooth.

We are now ready to establish the local LDP.
    Let $S_n := \sum_{k=1}^n B_{n,k}$, where $B_{n,k} \stackrel{d}{=} \textup{Ber}(\tfrac{1}{1+\lambda_{n,k}})$ are independent Bernoulli random variables for $k, n\in\mathbb{N}$ with $k\leq n$. By the representation \eqref{eq:polynomials}, we have $P_n(z) = \E z^{S_n}$ and $a_{n,k}=\P(S_n=k)$.

    First, we tilt the probability measure $\P$.
    More specifically, for a parameter $\theta>0$, we define a probability measure $\widetilde{\P}_\theta$ induced from $\P$ by
    \begin{equation*}
        \widetilde{\E}_\theta h(B_{n,1}, \ldots, B_{n,n}) = \frac{\E[h(B_{n,1}, \ldots, B_{n,n}) \theta^{S_n}]}{P_n(\theta)} \;,
    \end{equation*}
    for any bounded function $h:\{0,1\}^n\to\R$.
    Substituting $h(x_1, \ldots, x_n)=\mathbf{1}_{\{x_1 + \cdots +x_n = k\}}$, we have
    \begin{equation*}
        \P(S_n = k) = P_n(\theta) \theta^{-k} \widetilde{\P}_\theta(S_n = k).
    \end{equation*}
    Equivalently, we have
    \begin{align} \label{eq. step 1 proof JKM25}
        \frac{1}{c_n} \log \P(S_n=k) = - \frac{k}{c_n} \log \theta + \frac{1}{c_n} \log P_n(\theta) + \frac{1}{c_n} \log \widetilde{\P}_\theta(S_n=k).
    \end{align}
   Our goal is to tune the parameter $\theta\approx \Phi^{-1}(k/c_n)$ so that the first two terms of \eqref{eq. step 1 proof JKM25} give the desired result $g(x) = \Psi(\Phi^{-1}(x)) - x \log \Phi^{-1}(x)$ and the last term will be negligible.

    Let $b>a>0$ such that $[a,b]\subset (\underline{m}, \overline{m})$ is some compact interval, then for sufficiently large $n$,
    \begin{equation*}
        \theta_*(n,k) := \Phi_n^{-1}\Big( \frac{k}{c_n} \Big)
    \end{equation*}
    is well-defined for any $k \in \{0,\dots,n\} \cap [ac_n,bc_n] $.
    By locally uniform convergence of $\Phi_n^{-1}\to\Phi^{-1}$, we have
    \begin{align} \label{eq. theta* asymp Inverse Phi}
        \sup_{k \in \{0,\dots,n\} \cap [ac_n,bc_n] } \bigg|  \theta_*(n,k) -  \Phi^{-1}( \tfrac{k}{c_n }) \bigg| \xrightarrow{n\to\infty} 0.
    \end{align}
Combining \eqref{eq. theta* asymp Inverse Phi} with locally uniform convergence of $\Psi_n \to \Psi$, it follows that
    \begin{align} \label{eq. log Pn asymp}
	\sup_{k \in \{0,\dots,n\} \cap [ac_n,bc_n] } \bigg| \frac{1}{c_n}\log P_n(\theta_*(n,k)) - \Psi \Big( \Phi^{-1}( \tfrac{k}{c_n }) \Big) \bigg| \xrightarrow{n\to\infty} 0,
\end{align} 
and similarly,
    \begin{align}
        \sup_{k \in \{0,\dots,n\} \cap [ac_n,bc_n] } \bigg| \frac k {c_n}\log \theta_*(n,k) - \frac k {c_n}\log \Phi^{-1}( \tfrac{k}{c_n }) \bigg| \xrightarrow{n\to\infty} 0.
    \end{align}
    In particular, the first two terms of \eqref{eq. step 1 proof JKM25} approximate $g(k/c_n)$.

    By definition, $\widetilde{\P}_\theta(B_{n,j}=1) = \theta / (\theta+\lambda_{n,j})$ for $j= 1, \ldots, n$.
    Thus for $\theta_* \equiv \theta_*(n,k)$, we have
    \begin{align} \label{eq. def of sigma tilde}
        \widetilde{\sigma}_{n,k}^2:= \widetilde{\operatorname{Var}}_{\theta_*}(S_n) = \sum_{j=1}^n \frac{\lambda_{n,j} \theta_*}{(\theta_*+\lambda_{n,j})^2} = c_n\theta_* \int_{-\infty}^0 \frac{-y }{(\theta_*-y)^2} \mu_n(dy)= c_n \theta_*\Phi_n'(\theta_*).
    \end{align}
    We claim that
    \begin{align} \label{eq. boundedness of sigma n k}
        0 < \liminf_{n\to\infty} \inf_{k} \frac{1}{c_n} \widetilde{\sigma}_{n,k}^2 \leq \limsup_{n\to\infty} \sup_{k} \frac{1}{c_n} \widetilde{\sigma}_{n,k}^2 < \infty,
    \end{align}
    where again ${k \in \{0,\dots,n\} \cap [ac_n,bc_n] }$ for $[a,b]\subset (\underline{m}, \overline{m})$.
    The claim follows from the following argument.
    \begin{enumerate}
        \item  
        By the uniform convergence \eqref{eq. theta* asymp Inverse Phi}, we have $\theta_*(n,k) \in \Phi^{-1}([a/2,2b])$ for sufficiently large $n\in\mathbb{N}$ for all $k \in \{0,\dots,n\} \cap [ac_n,bc_n]$.
        In other words, for the same $n$ and $k$, the value $\theta_*(n,k)$ is bounded away from $0$ and $\infty$.
        \smallskip 
        \item We now show the upper bound in \eqref{eq. boundedness of sigma n k}.
        An elementary bound for integrands yields
        \begin{align*}
            \int_{-\infty}^0 \frac{-y}{(\theta_* -y)^2} \mu_n(dy) &\leq \int_{-\infty}^{-1} \frac{1}{\theta_*-y} \mu_n(dy) + \frac{1}{\theta_*} \int_{-1}^0 \frac{1}{\theta_* - y} \mu_n(dy)
            \leq \max\Big\{1, \frac{1}{\theta_*}\Big\}\Phi_n(\theta_*).
        \end{align*}
        Substituting it into \eqref{eq. def of sigma tilde}, the upper bound follows from the convergence of $\Phi_n\to\Phi$ and \eqref{eq. theta* asymp Inverse Phi}.
        \smallskip 
        \item We show the lower bound in \eqref{eq. boundedness of sigma n k}.
        Since $\Phi_n(z)$ and $\Phi(z)$ are strictly increasing in $z>0$ and $\Phi_n'\to\Phi'$ locally uniformly, we observe that
        \begin{equation*}
            \liminf_{n\to\infty}\inf_k\Phi_n'(\theta_*)\ge \inf_{x\in[a/2,2b]} \Phi'(x)>0.
        \end{equation*}
        Putting it into \eqref{eq. def of sigma tilde}, we obtain the lower bound.
    \end{enumerate}
    These prove the claim \eqref{eq. boundedness of sigma n k}.

    Since $c_n \to \infty$ as $n\to\infty$, the inequalities \eqref{eq. boundedness of sigma n k} implies $\tilde \sigma_{n,k} \to \infty$ for ${k \in \{0,\dots,n\} \cap [ac_n,bc_n] }$ as $n\to\infty$. Due to the divergence, we can apply the local limit theorem \cite[Theorem 2]{Be73} to $\widetilde{\P}_{\theta_*(n,k)}(x) := \widetilde{\E}_{\theta_*(n,k)} x^{S_n}$, hence
    \begin{align*}
        \sup_{x\in\R} \bigg| \widetilde{\sigma}_{n,k} \widetilde{\P}_{\theta_*(n,k)}\Big[S_n=\lfloor \widetilde{\sigma}_{n,k}x\rfloor + k \Big] - \frac{1}{\sqrt{2\pi}} e^{-x^2/2} \bigg| \xrightarrow{n\to\infty} 0.
    \end{align*}
    Substituting $x=0$, we obtain
    \begin{equation*}
        \widetilde{\P}_{\theta_*(n,k)}[S_n= k] \sim \frac{1}{\sqrt{2\pi} \widetilde{\sigma}_{n,k}}. 
    \end{equation*}
    Together with the uniformity of the asymptotics of $\tilde\sigma_{n,k}$ in $k$ from \eqref{eq. boundedness of sigma n k}, we conclude that
    \begin{align} \label{eq. log P tilde asymp 0}
        \sup_{k \in \{0,\dots,n\} \cap [ac_n,bc_n] } \bigg| \frac{1}{c_n}\log \widetilde{\P}_{\theta_*(n,k)}(S_n=k) \bigg| \xrightarrow{n\to\infty} 0.
    \end{align}
  Putting \eqref{eq. theta* asymp Inverse Phi}, \eqref{eq. log Pn asymp}, and \eqref{eq. log P tilde asymp 0} into \eqref{eq. step 1 proof JKM25}, we obtain \eqref{eq. exp profile of coeff}, since $a_{n,k} = \P(S_n=k)$.
\end{proof}

\subsection{Proof of Theorem~\ref{Thm_main}}
We are now ready to prove Theorem~\ref{Thm_main}.

\begin{proof}[Proof of Theorem~\ref{Thm_main}]
    We first prove Theorem~\ref{Thm_main}~(i), the strong asymmetry regime.
    Choose $c_n = \sqrt{2n}$ and let
    \begin{equation*}
     \Psi_n(z) := \frac{1}{c_n} \log\Big( \sum_{k=0}^n z^kp_{2n,2k} \Big), \qquad \Phi_n(z) := z\Psi_n'(z).
    \end{equation*}
    By Proposition~\ref{Prop_generating function limit}, the sequences $\Psi_n(z)$ converge to $\Psi_{\rm s}(z)$ in $\C \setminus(-\infty,0]$.
    Using the definitions \eqref{eq. def Psi sA} and \eqref{eq. def Phi sA}, it follows that $\Psi_{\rm s}$ and $\Phi_{\rm s}$ satisfy the conditions in Theorem~\ref{Thm. zero dist and exp profile} with 
    $\underline{m} = 0$ and $\overline{m} = \infty$.
    Indeed, by twice differentiation of \eqref{eq. integral rep of Polylog Li 3/2} it follows analogously to \eqref{eq:psi_s''} that    
    \begin{align*}
    \Big(\frac{d}{du}\Big)^2  \int_0^\infty \log \Big( 1- (1-e^{u}) e^{-t^2} \Big) \,dt = \int_0^\infty \frac{e^{u}e^{-t^2}(1-e^{-t^2})}{(1- (1-e^{u}) e^{-t^2})^2} \,dt >0,  
    \end{align*}
    hence $\Psi_{\rm s}(e^u)$ is strictly convex on $(0,\infty)$ and $\Phi$ is strictly monotonic.
    Then, the conclusion of Theorem~\ref{Thm_main}~(i) directly follows from Theorem~\ref{Thm. zero dist and exp profile} with $\phi_{\rm s}=-g$.

    For the proof of the weak asymmetry regime, Theorem~\ref{Thm_main}~(ii), we choose $c_n = 2n$.
    Then, the above arguments again work after replacing $\Psi_{\rm s}$, $\Phi_{\rm s}$, and $\overline{m} = \infty$ by $\Psi_{\rm w}$, $\Phi_{\rm w}$, and $\overline{m} = 1$, respectively and recalling that convexity of $\Psi_{\rm w}(e^z)$ follows from \eqref{eq:psi_s''}.
    This finishes the proof.
\end{proof}

We conclude this section with a brief comment on the obstructions that arise in the odd-dimensional eGinOE. We begin by recalling the key ingredients used in the even-dimensional case.

Throughout the proof, the determinant formula \eqref{eq for det formula of gen function} plays a crucial role in the analysis. It reduces the study of the generating function of $p_{2n,2k}$ to the analysis of trace powers of a matrix $M_n$ via Lemma~\ref{lem. trace power expansion of the generating function}. Moreover, this representation implies that the generating function has only negative real zeros, a property that is subsequently exploited to apply Theorem~\ref{Thm. zero dist and exp profile}.

In the odd-dimensional case, however, no analogue of the determinant formula \eqref{eq for det formula of gen function} is currently available. While the generating function of $p_{2n+1,2k+1}$ admits a Pfaffian representation, it does not appear to reduce to a determinant in a tractable way, see \cite[Theorem 2.1]{Si07} and \cite[Proposition 1]{FN08}. In particular, the matrix involved lacks the checkerboard structure that enables such a reduction in the even-dimensional setting. Consequently, the strategy employed above cannot be directly extended to the odd-dimensional eGinOE.

 \section{Conclusion}

In this paper, we have investigated the statistics of the total number of real eigenvalues of random matrices belonging to the elliptic real Ginibre ensemble of size $n \times n$. We have analyzed both the weak asymmetry regime (corresponding to $1 - \tau = O(1/n)$) and the strong asymmetry regime (for fixed $\tau \in [0,1)$). In both cases, we derived the asymptotic form of the probability $p_{n,m}$ that a matrix possesses exactly $m$ real eigenvalues in the large $n$ limit.  

In the weak asymmetry regime, we focused on the scaling $m = O(n)$ and showed that  
\[
p_{n,m} \sim \exp[-n \, \phi_{\mathrm{w}}(m/n;\alpha)] \, ,
\]
where the rate function $\phi_{\mathrm{w}}(x;\alpha)$ was computed explicitly. This describes the full large deviation behaviour in this regime and, in particular, contains as a special case the Gaussian fluctuations corresponding to the minimum of $\phi_{\mathrm{w}}(x;\alpha)$. In the strong asymmetry regime, we instead considered the scaling $m = O(\sqrt{n})$ and obtained  
\[
p_{n,m} \sim \exp[-\sqrt{n} \, \phi_{\mathrm{s}}(m/\sqrt{n};\tau)] \, ,
\]
as stated in Theorem~\ref{Thm_main}(i). Here again, the minimum of the rate function $\phi_{\mathrm{s}}(x;\tau)$ corresponds to typical Gaussian fluctuations. However, in contrast to the strongly asymmetric case, this regime coexists with an additional large deviation regime $m = O(n)$, previously studied in \cite{MPTW16} using a Coulomb gas approach. Understanding precisely how these two regimes match and interpolate remains a challenging open problem.

Our results naturally give rise to several further questions. 
First, while we have focused on the total number of real eigenvalues along the entire real axis, it would be interesting to study the statistics of the number of real eigenvalues contained in a finite interval $[a,b] \subset \mathbb{R}$. 
One could, for example, examine how the cumulants of this local counting variable depend on the length or position of the interval; see \cite{ABES23,Fo15a} for a related discussion in the symmetric interval case. 
Second, it would be natural to explore the joint statistics of real and complex eigenvalues---for instance, the number of complex eigenvalues lying in a region of the complex plane that intersects the real axis.

Finally, our analysis relied on extending the recent techniques developed in \cite{JKM25, JKM25II} to extract large deviation form from the analysis of associated generating functions.  It would be interesting to explore whether similar methods can be applied to study large deviations and rate functions in other models of integrable probability. For instance, large-deviation forms akin to \eqref{eq. Gregory's ansatz} were conjectured for the distribution of the number of real roots of real random polynomials~\cite{SM07,SM08,PS18}. Notably, several parallels have been observed \cite{SM08} between the real roots of the so-called real Weyl polynomials and the real eigenvalues of the real Ginibre ensemble analyzed here. It would therefore be natural to investigate whether the methods developed in the present work can be extended to the study of real random polynomials and related problems~\cite{PS18}.

%
%

\begin{acks}[Acknowledgments]
The authors would like to thank Peter Forrester, Paul Krapivsky and Mihail Poplavskyi for useful discussions. 
\end{acks}
\begin{funding}
SSB was supported by the National Research Foundation of Korea grants (RS-2023-00301976, RS-2025-00516909). JJ was supported by the DFG priority program SPP 2265 Random Geometric Systems. GS was supported by ANR Grant No. ANR-23-CE30-0020-01 EDIPS.
\end{funding}

\end{document}